\documentclass[12pt]{amsart}
\usepackage{amsmath}
\usepackage{amsfonts}
\usepackage{amssymb}
\usepackage{bbding}
\usepackage{txfonts}
\usepackage[shortlabels]{enumitem}
\usepackage{ifpdf}
\ifpdf
  \usepackage[colorlinks,final,backref=page,hyperindex]{hyperref}
\else
  \usepackage[colorlinks,final,backref=page,hyperindex,hypertex]{hyperref}
\fi

\newtheorem{theorem}{Theorem}[section]
\newtheorem{lemma}[theorem]{Lemma}
\newtheorem{coro}[theorem]{Corollary}

\newtheorem{conjecture}[theorem]{Conjecture}
\newtheorem{prop}[theorem]{Proposition}
\theoremstyle{definition}
\newtheorem{defn}[theorem]{Definition}

\newtheorem{exam}[theorem]{Example}
\newtheorem{algorithm}[theorem]{Algorithm}

\newcommand{\nc}{\newcommand}
\newcommand{\delete}[1]{}

\nc{\tred}[1]{\textcolor{red}{#1}} \nc{\tblue}[1]{\textcolor{blue}{#1}} \nc{\tgreen}[1]{\textcolor{green}{#1}} \nc{\tpurple}[1]{\textcolor{purple}{#1}} \nc{\btred}[1]{\textcolor{red}{\bf #1}} \nc{\btblue}[1]{\textcolor{blue}{\bf #1}} \nc{\btgreen}[1]{\textcolor{green}{\bf #1}} \nc{\btpurple}[1]{\textcolor{purple}{\bf #1}}

\renewcommand{\Bbb}{\mathbb}

\newcommand{\efootnote}[1]{}

\renewcommand{\textbf}[1]{}

\nc{\mlabel}[1]{\label{#1}}  
\nc{\mcite}[2][]{\cite[#1]{#2}}  
\nc{\mref}[1]{\ref{#1}}  
\nc{\mbibitem}[1]{\bibitem{#1}} 

\delete{
\nc{\mlabel}[1]{\label{#1}  
{\hfill \hspace{1cm}{\bf{{\ }\hfill(#1)}}}}
\nc{\mcite}[2][1]{\cite[#1]{#2}{{\bf{{\ }(#2; #1)}}}}  
\nc{\mref}[1]{\ref{#1}{{\bf{{\ }(#1)}}}}  
\nc{\mbibitem}[1]{\bibitem[\bf #1]{#1}} 
}

\renewcommand\geq{\geqslant}
\renewcommand\leq{\leqslant}

\renewcommand\bar[1]{\overline{#1}}

\nc{\into}{I}
\nc{\rbw}{\mathfrak{R}} \nc{\brp}{\mathrm{brp}} \nc{\lead}{\mathrm{Lead}} \nc{\Id}{\mathrm{Id}} \nc{\Irr}{\mathrm{Irr}}
\nc{\vx}{\sigma} \nc{\vy}{\tau} \nc{\dvx}{\sigma^{(1)}} \nc{\dvy}{\tau^{(1)}} \nc{\done}{\vep} \nc{\mcitep}[1]{\mcite{#1}} \nc{\wt}{\mathrm{wt}} \nc{\bre}[1]{|#1|} \nc{\mapmonoid}{\frakM} \nc{\disjoint}{\frakM'}
\nc{\ncpoly}[1]{\langle #1\rangle}  
\nc{\mapm}[1]{\lfloor\!|{#1}|\!\rfloor}
\nc{\diff}[1]{{}^\NC\{ #1 \}} \nc{\disj}[1]{\{{#1}\}'} \nc{\mdisj}[1]{\frakM'(#1)} \nc{\brho}{\bar{\rho}} \nc{\om}{\bar{\frakm}} \nc{\frakn}{\mathfrak n} \nc{\ddeg}[1]{^{(#1)}} \nc{\opset}{X} \nc{\genset}{{Z}} \nc{\NC}{\mathrm{{NC}}} \nc{\leaf}{\mathrm{leaf}} \nc{\twig}{\mathrm{twig}} \nc{\fe}{\mathrm{fl}} \nc{\munderline}[1]{#1} \nc{\bo}{o} \nc{\dep}{\mathrm{depth}} \nc{\ofe}{\mathrm{ofl}} \nc{\dfe}{\mathrm{dfe}} \nc{\fex}{\mathrm{fex}} \nc{\dl}{\mathrm{dlex}} \nc{\db}{\mathrm{db}} \nc{\lex}{\mathrm{lex}} \nc{\clex}{\mathrm{clex}} \nc{\dgp}{\mathrm{dgp}} \nc{\dgx}{\mathrm{dgx}} \nc{\br}{\mathrm{br}} \nc{\obd}{\mathrm{odb}} \nc{\ob}{\mathrm{ob}}
\nc{\pie}{\mathrm{PIE}}
\nc{\rbo}{\mathrm{RBO}}
\nc{\supp}{\mathcal{S}}
\nc{\nul}{\mathcal{Z}}

\nc{\bin}[2]{ (_{\stackrel{\scs{#1}}{\scs{#2}}})}  
\nc{\binc}[2]{ \left (\!\! \begin{array}{c} \scs{#1}\\
    \scs{#2} \end{array}\!\! \right )}  
\nc{\bincc}[2]{  \left ( {\scs{#1} \atop
    \vspace{-1cm}\scs{#2}} \right )}  
\nc{\bs}{\bar{S}} \nc{\cosum}{\sqsubset} \nc{\la}{\longrightarrow} \nc{\rar}{\rightarrow} \nc{\dar}{\downarrow} \nc{\dprod}{**} \nc{\dap}[1]{\downarrow \rlap{$\scriptstyle{#1}$}} \nc{\md}[1]{\bar{#1}} \nc{\uap}[1]{\uparrow \rlap{$\scriptstyle{#1}$}} \nc{\defeq}{\stackrel{\rm def}{=}} \nc{\disp}[1]{\displaystyle{#1}} \nc{\dotcup}{\ \displaystyle{\bigcup^\bullet}\ } \nc{\gzeta}{\bar{\zeta}} \nc{\hcm}{\ \hat{,}\ } \nc{\hts}{\hat{\otimes}} \nc{\barot}{{\otimes}} \nc{\free}[1]{\bar{#1}} \nc{\uni}[1]{\tilde{#1}} \nc{\hcirc}{\hat{\circ}} \nc{\leng}{\ell} \nc{\lleft}{[} \nc{\lright}{]} \nc{\lc}{\lfloor} \nc{\rc}{\rfloor}
\nc{\lb}{[} 
\nc{\rb}{]} 
\nc{\curlyl}{\left \{ \begin{array}{c} {} \\ {} \end{array}
    \right.  \!\!\!\!\!\!\!}
\nc{\curlyr}{ \!\!\!\!\!\!\!
    \left. \begin{array}{c} {} \\ {} \end{array}
    \right \} }
\nc{\longmid}{\left | \begin{array}{c} {} \\ {} \end{array}
    \right. \!\!\!\!\!\!\!}
\nc{\onetree}{\bullet} \nc{\ora}[1]{\stackrel{#1}{\rar}}
\nc{\ola}[1]{\stackrel{#1}{\la}}
\nc{\ot}{\otimes} \nc{\mot}{{{\boxtimes\,}}} \nc{\otm}{\overline{\boxtimes}} \nc{\sprod}{\bullet} \nc{\scs}[1]{\scriptstyle{#1}} \nc{\mrm}[1]{{\rm #1}} \nc{\msum}{\sum\limits}
\nc{\margin}[1]{\marginpar{\rm #1}}   
\nc{\dirlim}{\displaystyle{\lim_{\longrightarrow}}\,} \nc{\invlim}{\displaystyle{\lim_{\longleftarrow}}\,} \nc{\mvp}{\vspace{0.3cm}} \nc{\tk}{^{(k)}} \nc{\tp}{^\prime} \nc{\ttp}{^{\prime\prime}} \nc{\svp}{\vspace{2cm}} \nc{\vp}{\vspace{8cm}} \nc{\proofbegin}{\noindent{\bf Proof: }}
\nc{\proofend}{$\blacksquare$ \vspace{0.3cm}}
\nc{\modg}[1]{\!<\!\!{#1}\!\!>}
\nc{\intg}[1]{F_C(#1)} \nc{\lmodg}{\!<\!\!} \nc{\rmodg}{\!\!>\!} \nc{\cpi}{\widehat{\Pi}}
\nc{\sha}{{\mbox{\cyr X}}}  
\nc{\shap}{{\mbox{\cyrs X}}} 
\nc{\shpr}{\diamond}    
\nc{\shp}{\ast} \nc{\shplus}{\shpr^+}
\nc{\shprc}{\shpr_c}    
\nc{\msh}{\ast} \nc{\zprod}{m_0} \nc{\oprod}{m_1} \nc{\vep}{\varepsilon} \nc{\labs}{\mid\!} \nc{\rabs}{\!\mid}
\nc{\astarrow}{\overset{\raisebox{-3pt}{$\ast$}}{\rightarrow}}

\nc{\dth}{d} \nc{\mmbox}[1]{\mbox{\ #1\ }} \nc{\fp}{\mrm{FP}} \nc{\rchar}{\mrm{char}} \nc{\Fil}{\mrm{Fil}} \nc{\Mor}{Mor\xspace} \nc{\gmzvs}{gMZV\xspace} \nc{\gmzv}{gMZV\xspace} \nc{\mzv}{MZV\xspace} \nc{\mzvs}{MZVs\xspace} \nc{\Hom}{\mrm{Hom}} \nc{\id}{\mrm{id}} \nc{\im}{\mrm{im}} \nc{\incl}{\mrm{incl}} \nc{\map}{\mrm{Map}} \nc{\mchar}{\rm char} \nc{\nz}{\rm NZ}

\nc{\Alg}{\mathbf{Alg}} \nc{\Bax}{\mathbf{Bax}} \nc{\bff}{\mathbf f} \nc{\bfk}{{\bf k}} \nc{\bfone}{{\bf 1}} \nc{\bfx}{\mathbf x} \nc{\bfy}{\mathbf y}
\nc{\base}[1]{\bfone^{\otimes ({#1}+1)}} 
\nc{\Cat}{\mathbf{Cat}} \delete{}
\nc{\detail}{\marginpar{\bf More detail}
    \noindent{\bf Need more detail!}
    \svp}
\nc{\Int}{\mathbf{Int}} \nc{\Mon}{\mathbf{Mon}}
\nc{\rbtm}{{shuffle }} \nc{\rbto}{{Rota-Baxter }} \nc{\remarks}{\noindent{\bf Remarks: }} \nc{\Rings}{\mathbf{Rings}} \nc{\Sets}{\mathbf{Sets}}

\nc{\BA}{{\Bbb A}} \nc{\CC}{{\Bbb C}} \nc{\DD}{{\Bbb D}} \nc{\EE}{{\Bbb E}} \nc{\FF}{{\Bbb F}} \nc{\GG}{{\Bbb G}} \nc{\HH}{{\Bbb H}} \nc{\LL}{{\Bbb L}} \nc{\NN}{{\Bbb N}} \nc{\KK}{{\Bbb K}} \nc{\QQ}{{\Bbb Q}} \nc{\RR}{{\Bbb R}} \nc{\TT}{{\Bbb T}} \nc{\VV}{{\Bbb V}} \nc{\ZZ}{{\Bbb Z}}

\nc{\cala}{{\mathcal A}} \nc{\calc}{{\mathcal C}} \nc{\cald}{{\mathcal D}} \nc{\cale}{{\mathcal E}} \nc{\calf}{{\mathcal F}} \nc{\calg}{{\mathcal G}} \nc{\calh}{{\mathcal H}} \nc{\cali}{{\mathcal I}} \nc{\call}{{\mathcal L}} \nc{\calm}{{\mathcal M}} \nc{\caln}{{\mathcal N}} \nc{\calo}{{\mathcal O}} \nc{\calp}{{\mathcal P}} \nc{\calr}{{\mathcal R}} \nc{\cals}{{\mathcal S}} \nc{\calt}{{\mathcal T}} \nc{\calw}{{\mathcal W}} \nc{\calk}{{\mathcal K}} \nc{\calx}{{\mathcal X}}
\nc{\calz}{{\mathcal Z}}
 \nc{\CA}{\mathcal{A}}

\nc{\fraka}{{\mathfrak a}} \nc{\frakA}{{\mathfrak A}} \nc{\frakb}{{\mathfrak b}} \nc{\frakB}{{\mathfrak B}} \nc{\frakD}{{\mathfrak D}} \nc{\frakH}{{\mathfrak H}} \nc{\frakM}{{\mathfrak M}} \nc{\bfrakM}{\overline{\frakM}} \nc{\frakm}{{\mathfrak m}} \nc{\frakP}{{\mathfrak P}} \nc{\frakN}{{\mathfrak N}} \nc{\frakp}{{\mathfrak p}} \nc{\frakS}{{\mathfrak S}} \nc{\frakx}{{\mathfrak x}} \nc{\ox}{\bar{\frakx}} \nc{\frakX}{{\mathfrak X}} \nc{\fraky}{{\mathfrak y}} \nc\dop{\delta}
\nc{\Reduce}{{\rm Red}}

\font\cyr=wncyr10 \font\cyrs=wncyr7
\nc{\redt}[1]{\textcolor{red}{#1}}
\nc{\ma}[1]{\textcolor{green}{\tt Markus:#1}}
\nc{\li}[1]{\textcolor{red}{\tt Li:#1}} \nc{\sz}[1]{\textcolor{blue}{\tt sz:#1}} \nc{\xg}[1]{\textcolor{purple}{\tt xg:#1}}

\nc{\nonz}[1]{#1^\times}
\nc{\NNP}{\nonz{\NN}}
\nc{\stdint}[1]{J_{#1}}
\nc{\dualmod}[1]{#1^*}
\nc{\alghom}[1]{#1^\bullet}
\nc{\End}{\mathrm{End}}
\nc{\rng}{\mathrm{\mathcal{R}}}
\nc{\codim}{\mathrm{codim}}
\nc{\evl}{\mathrm{ev}}

\newenvironment{thmenumerate}{\leavevmode\begin{enumerate}[leftmargin=1.5em]}{\end{enumerate}}

\begin{document}
\title[Rota-Baxter operators on the polynomial algebra]{Rota-Baxter operators on the polynomial algebras, integration and averaging operators}

\author{Li Guo}
\address{
Department of Mathematics and Computer Science, Rutgers University, Newark, NJ 07102, USA}
\email{liguo@rutgers.edu}

\author{Markus Rosenkranz}
\address{
    School of Mathematics, Statistics and Actuarial Science,
    University of Kent,
    Canterbury CT2 7NF, England}
\email{M.Rosenkranz@kent.ac.uk}

\author{Shanghua Zheng}
\address{Department of Mathematics, Lanzhou University, Lanzhou, Gansu 730000, China}
\email{zheng2712801@163.com}

\hyphenpenalty=8000
\date{\today}

\begin{abstract}
  Rota-Baxter operators are an algebraic abstraction of integration. Following this classical connection, we study the relationship between Rota-Baxter operators and integrals in the case of the polynomial algebra $\bfk[x]$. We consider two classes of Rota-Baxter operators, monomial ones and injective ones. For the first class, we apply averaging operators to determine monomial Rota-Baxter operators. For the second class, we make use of the double product on Rota-Baxter algebras.
\end{abstract}

\subjclass[2010]{16W99, 45N05, 47G10, 12H20}

\keywords{Rota-Baxter operator, averaging operator, integration, monomial linear operator}

\maketitle

\tableofcontents

\hyphenpenalty=8000 \setcounter{section}{0}


\section{Introduction}\mlabel{sec:int}

Rota-Baxter operators are deeply rooted in analysis. Their study originated from
the work of G.~Baxter~\cite{Ba} in 1960 on Spitzer's identity~\cite{Sp} in
fluctuation theory. More fundamentally, the notion of Rota-Baxter operator is an
algebraic abstraction of the \emph{integration by parts formula} of
calculus. Throughout the 1960s, Rota-Baxter operators were studied by well-known
analysts such as Atkinson~\cite{At}. In the 1960s and 1970s, the works of Rota
and Cartier~\cite{Ca,Ro1} led the study of Rota-Baxter operators into algebra
and combinatorics. In the 1980s, the Rota-Baxter operator for Lie algebras was
independently discovered by mathematical physicists as the operator form of the
classical Yang-Baxter equation~\cite{STS}. In the late 1990s, the operator
appeared again as a fundamental algebraic structure in the work of Connes and
Kreimer on renormalization of quantum field theory~\cite{CK}. The present
century witnesses a remarkable renaissance of Rota-Baxter operators through
systematic algebraic studies with wide applications to combinatorics, number
theory, operads and mathematical
physics~\cite{Ag,Bai,BBGN,BGN,CK,EGK,EGM,GK1,GZ}. See~\cite{Guw} for a brief
introduction and~\cite{Gub} for a more detailed treatment.

Recently, Rota-Baxter operator related structures, including differential
Rota-Baxter algebras~\cite{GK3} and \emph{integro-differential
  algebras}~\cite{RR}, were introduced in the algebraic study of calculus,
especially in boundary problems for linear differential
equations~\cite{GGZ,GRR}. The upshot is that the Green's operator of such a
boundary problem can be represented by suitable operator rings based on an
integro-differential algebra.

In this paper, we revisit the analysis origin of Rota-Baxter operators to study
how their algebraic properties are linked with their analytic appearance. We
focus on the polynomial algebra $\RR[x]$, which plays a central role both in
analysis where it is taken as approximation of analytic functions, and in
algebra where it is the free object in the category of commutative
algebras. This algebra, together with the standard integral operator, is also
the free commutative Rota-Baxter algebra on the empty set or, in other words,
the initial object in the category of commutative Rota-Baxter algebras. Thus it
provides an ideal testing ground for the interaction between analytically
defined Rota-Baxter operators and the algebraically defined Rota-Baxter
operators.

One natural question in this regard is when an algebraically defined Rota-Baxter
operator on $\RR[x]$ can be realized in analysis. It is a classical fact that
the Riemann integral with variable upper limit is a Rota-Baxter operator of
weight zero on $\RR[x]$. This remains true when the integral operator is
pre-multiplied by any polynomial. We might call these Rota-Baxter operators on
$\RR[x]$ \emph{analytically modelled}. It is easy to see that such operators are
injective. We conjecture that all injective Rota-Baxter operators on $\RR[x]$
are indeed analytically modelled. We provide evidence for this conjecture by
exploring two classes of such operators.

The first comprises what we call \emph{monomial Rota-Baxter operators} over an
arbitrary integral domain~$\bfk$ of characteristic zero, meaning Rota-Baxter
operators $P$ with $P(x^n) = ax^k$, where both $a\in\bfk$ and $k\in\NN$ may
depend on $n$. We classify monomial Rota-Baxter operators on $\bfk[x]$ and show
that all injective monomial Rota-Baxter operators are analytically modelled. The
second class is restricted to~$\bfk=\RR$ and contains those operators that
satisfy a \emph{differential law}~$\partial \circ P = r$, where the right-hand
side denotes the multiplication operator induced by an arbitrarily monomial~$r
\in \RR[x]$. We show that any injective Rota-Baxter operator is of this form
and, provided~$r$ is monomial, analytically modelled.

In Section~\mref{sec:basis} we discuss general algebraic properties of
Rota-Baxter operators that will be used in subsequent sections. In
Section~\mref{sec:mono} we focus on monomial Rota-Baxter operators. While
determining these operators, we prove that all injective monomial Rota-Baxter
operators are analytically modelled. In Section~\mref{sec:inj}, we study
injective Rota-Baxter operators in general (on the real polynomial ring). We
first show that injective Rota-Baxter operators are precisely those that satisfy
a differential law. Then we prove that, in the monomial case, they are
analytically modelled.


\section{General concepts and properties}
\mlabel{sec:basis}

\paragraph{\bf Notation} If~$M$ is a monoid we write~$\nonz{M} = \{ x
\in M \mid x \ne 0_M\}$ for the semigroup of nonzero elements. In
particular, the monoid of natural numbers (nonnegative integers) is
denoted by~$\NN$, so~$\NNP$ is the semigroup of positive integers. The
notation~$l \mid k$ signifies that~$l$ is a divisor of~$k$.

We use~$\bfk$ to denote a commutative ring with identity~$1$ unless otherwise
specified. All $\bfk$-algebras in this paper are assumed to be commutative and
with a unit~$1_A$ that will be identified with~$1_\bfk$ through the structure
map~$\bfk \to A$.

We start by collecting some general properties of Rota-Baxter
operators for later use. First we give the definition of a Rota-Baxter
$\bfk$-algebra of arbitrary weight~\cite{Ba,Gub,Ro2}.

\begin{defn}
  Let $\lambda$ be a given element of $\bfk$. A {\bf Rota-Baxter
    $\bfk$-algebra of weight $\lambda$}, or simply an {\bf RBA of
    weight $\lambda$}, is a pair $(R,P)$ consisting of a
  $\bfk$-algebra $R$ and a linear operator $P\colon R\to R$ that
  satisfies the {\bf Rota-Baxter equation}
  \begin{equation}
    P(u)P(v)=P(uP(v))+P(P(u)v)+\lambda P(uv),\quad \text{for all}\, u,v\in R.
    \mlabel{eq:RB}
  \end{equation}
  Then $P$ is called a {\bf Rota-Baxter operator of weight
    $\lambda$}. If $R$ is only assumed to be a nonunitary
  $\bfk$-algebra, we call $R$ a nonunitary Rota-Baxter $\bfk$-algebra
  of weight $\lambda$.
\end{defn}

Observe first that the standard integration operator~$\stdint{0}\colon
\bfk[x] \to \bfk[x]$, given by~$x^n \mapsto x^{n+1}/(n+1)$, is a
(prototypical) Rota-Baxter operator of weight~$0$. Of course the
choice of initialization point is irrelevant, so for any~$a \in \bfk$
there is another weight~$0$ Rota-Baxter operator~$\stdint{a}\colon
\bfk[x] \to \bfk[x]$, given by~$x^n \mapsto (x^{n+1} -
a^{n+1})/(n+1)$. In this paper we shall only be concerned with
the weight~$0$ case, so from now on the term ``Rota-Baxter operator''
is to be understood as ``Rota-Baxter operator of weight~$0$''.

Recall that from a derivation~$\delta$ on a commutative
$\bfk$-algebra~$R$ one can produce a new derivation~$r\delta$ by
postmultiplying with any~$r \in R$. Analogously, from a Rota-Baxter
operator~$P$ on~$R$ one obtains a new Rota-Baxter operator~$Pr$ by
\emph{premultiplying} with any~$r \in R$. Indeed, we have
\begin{equation*}
  (Pr)(u) \, (Pr)(v) = P(ru) \, P(rv) = P(ru \, P(rv)) + P(P(ru)\,
  rv) = (Pr)(u \, (Pr)(v)) + (Pr)((Pr)(u) \, v)
\end{equation*}
for any $u,v\in R$. Applying this to~$R = \bfk[x]$, we obtain the
family~$\stdint{a} r$ of \emph{analytically modelled} Rota-Baxter operators
on~$\bfk[x]$, where~$a \in \bfk$ and~$r \in \bfk[x]$ are arbitrary. As we will
show in Theorem~\ref{thm:main}, in the case of monomials~$r$, this family
exhausts the injective Rota-Baxter operators.

Let $\End(R) := \End_\bfk(R)$ denote the $\bfk$-module of linear operators on
$R$. Then the subset $\rbo(R)$ of $\End(R)$ consisting of Rota-Baxter operators
$P\colon R\to R$ is closed under \emph{multiplications by scalars} $c \in \bfk$
since in that case~$Pc = c P$. In the case of derivations on~$R$ more is true
since they form a $\bfk$-module (in fact a Lie algebra) while in general the sum
of two Rota-Baxter operators is not a Rota-Baxter operator. This motivates the
following terminology.

\begin{defn}
\begin{thmenumerate}
\item We call two Rota-Baxter operators $P_1, P_2 \in \rbo(R)$ {\bf
    compatible} if $c_1P_1+c_2P_2$ are in $\rbo(R)$ for
  all $c_1, c_2\in \bfk$.
\item Let $P \in \rbo(R)$. Then $Q \in \End(R)$ is called {\bf
    consistent} with~$P$ if~$P-Q$ is in $\rbo(R)$.
\item For~$P, Q \in \End(R)$ we define the bilinear
  form~$RB(P,Q)\colon R \otimes R \to R$ by
  \begin{equation*}
    RB(P,Q)(u,v):=P(u)Q(v)-P(uQ(v))-Q(P(u)v), \quad u, v\in R.
  \end{equation*}
  Thus $P \in \rbo(R)$ means that $RB(P,P)=0$ on $R\otimes R$.
\end{thmenumerate}
\end{defn}

Recall that for a Rota-Baxter algebra $(R,P)$, the multiplication
\begin{equation*}
  \star_P: R\otimes R \to R, \quad u\star_P v: = P(u)v+uP(v)  \text{ for all } u, v\in R,
\end{equation*}
is an associative product on $R$, called the {\bf double
  multiplication}~\mcite[Thm.~1.1.17]{Gub}. Moreover, $P\colon (R, \star_P) \to
R$ is then a homomorphism of nonunitary Rota-Baxter algebras.

If~$A$ is a $\bfk$-module, its (linear) \emph{dual} is denoted
by~$\dualmod{A}$. If~$A$ is moreover a $\bfk$-algebra, we use the notation
\begin{equation*}
  \alghom{A} := \{ \phi \in \dualmod{A} \mid  \phi(uv) = \phi(u)
  \phi(v)\}
\end{equation*}
for the set of \emph{multiplicative functionals}. Through the structure
map~$\bfk \to A$ we may also view the elements of~$\dualmod{A}$ as $\bfk$-linear
operators from~$A$ to~$\bfk$, and those of~$\alghom{A}$ as $\bfk$-algebra
homomorphisms from~$A$ to~$\bfk$.

\begin{prop}
\begin{thmenumerate}
\item Two Rota-Baxter operators $P_1, P_2 \in \rbo(R)$ are compatible
  if and only if $RB(P_1,P_2)+RB(P_2,P_1)=0$. This will be the case in
  particular when
\begin{equation*}
  \qquad
  P_1(u)P_2(v)=P_1(u P_2(v))+P_2(P_1(u)v) \quad\text{and}\quad
  P_2(u)P_1(v)=P_2(u P_1(v))+P_1(P_2(u)v)
\end{equation*}
holds for all~$u, v \in R$.
\mlabel{it:comp}
\item Let~$P \in \rbo(R)$ and~$Q \in \End(R)$ be given. Then~$Q$ is
  consistent with~$P$ if and only if
\begin{equation*}
RB(Q,Q) = RB(P,Q)+RB(Q,P).
\end{equation*}
is satisfied.
\mlabel{it:cons}
\item Let $P$ be in $\rbo(R)$. The set of $f \in \dualmod{R}$ that are consistent with~$P$ equals  $\alghom{(R,\star_P)}$.
  \mlabel{it:fun}
\end{thmenumerate}
\mlabel{pp:compcons}
\end{prop}

\begin{proof}
  (\mref{it:comp}) For arbitrary~$c_1, c_2 \in \bfk$, the
  bilinear form~$RB(c_1P_1+c_2P_2,c_1P_1+c_2P_2)$
  is given by
$$c_1^2 RB(P_1,P_1)+c_1c_2 (RB(P_1,P_2)+RB(P_2,P_1))+c_2^2 RB(P_2,P_2),$$
which simplifies to $c_1c_2 (RB(P_1,P_2)+RB(P_2,P_1))$ since
$P_1, P_2 \in \rbo(R)$.\smallskip

\noindent
(\mref{it:cons}) Since~$P \in \rbo(R)$ we have
$$RB(P-Q,P-Q)=-RB(P,Q)-RB(Q,P)+RB(Q,Q),$$
and hence the conclusion.\smallskip

\noindent
(\mref{it:fun}) Using that $P$ is a linear operator and $f$ a linear
functional, we have
$$ RB(f,f)=-f(u)f(v), \quad RB(f,P)(u,v)=-f(uP(v)), \quad RB(P,f)(u,v)=-f(P(u)v).$$
Thus by Item~(\mref{it:cons}) we conclude that~$f$ is consistent
with~$P$ if and only if
$$f(u)f(v)=f\big(P(u)v+uP(v)\big) = f(u\star_Pv),$$
which is what we need.
\end{proof}

\section{Monomial Rota-Baxter operators on
  \texorpdfstring{$\bfk[x]$}{bfkx}}
\label{sec:mono}

In this section, we determine the Rota-Baxter operators on $\bfk[x]$ that send
monomials to monomials and determine the analytically modelled ones. Throughout
this section, we assume that~$\bfk$ is an integral domain containing~$\QQ$.

\subsection{General properties}
We first give general criteria for a monomial linear operator to be a
Rota-Baxter operator before specializing in the following sections to the two
cases of nondegenerate and degenerate operators.

\begin{defn}
\begin{thmenumerate}
\item A linear operator $P$ on $\bfk[x]$ is called {\bf monomial} if
  for each $n\in \NN$, we have
\begin{equation}
P(x^n)=\beta(n)x^{\theta(n)}\qquad\text{with}\qquad\beta\colon\NN\to\bfk\quad\text{and}\quad\theta\colon\NN\to\NN.
\mlabel{eq:mono}
\end{equation}
\noindent If $\beta(n)=0$, the value of~$\theta(n)$ does not matter;
by convention we set $\theta(n)=0$ in this case.

\item
A monomial operator is called {\bf degenerate} if $\beta(n) = 0$
for some $n \in \NN$.
\end{thmenumerate}
\mlabel{defn:monrb}
\end{defn}

Let $A $ be a nonempty set and let $B$ be a set containing a distinguished
element $0$. For a map~$\phi\colon A\to B$ we define its {\bf zero set} as
$\nul_\phi:=\{a\in A\,|\, \phi(a)=0\}$ to be the {\bf zero set} of $\phi$.  Then
and its {\bf support} as $\supp_\phi:=A\setminus \nul_\phi$. Thus a monomial linear
operator~$P$ on~$\bfk[x]$ is nondegenerate if and only if
$\nul_\beta=\emptyset$.  As the following lemma shows, for a Rota-Baxter
operator $P$, degeneracy at~$n \in \NN$ occurs whenever~$P$ is constant on the
corresponding monomial.

\begin{lemma}
  Let $P$ be a monomial Rota-Baxter operator on $\bfk[x]$ and let $n \in \NN$. If $P(x^n)$ is in $\bfk$,
  then $P(x^n)=0$. In other words, $\supp_\beta= \supp_\theta$, and hence $\nul_\beta=\nul_\theta$.
\mlabel{lem:zero}
\end{lemma}
\begin{proof}
  If $P(x^n)=c$ is a nonzero constant, we have
  $$P(x^n)P(x^n)=c^2\neq 2c^2 =2 P(x^nP(x^n)).$$
  Hence $P$ is not a Rota-Baxter operator, and we must have~$c=0$.
\end{proof}

\begin{theorem} Let $P$ be a monomial linear operator on $\bfk[x]$ defined by $P(x^n)=\beta(n)x^{\theta(n)},$ $n \in \NN$. Then $P$ is a Rota-Baxter operator if $\theta$ and  $\beta$ satisfy the following conditions
\begin{thmenumerate}
\item
$\nul_\beta+\theta(\supp_\beta)\subseteq \nul_\beta$\,;
\mlabel{it:gen1}
\item
We have
\begin{align}
&\theta(m)+\theta(n)=\theta(
m+\theta(n))=\theta(\theta(
m)+n),
\mlabel{eq:dege1}\\
&\beta(m)\beta(n)=\beta(m
+\theta(n))\beta(n)+\beta(
n+\theta(m))\beta(m),
\mlabel{eq:dege2}
\end{align}
for all $m,n\in \supp_\beta$\,.
\mlabel{it:gen2}
\end{thmenumerate}
Under the assumption that $\supp_\beta+\theta(\supp_\beta)\subseteq \supp_\beta$, if $P$ is a Rota-Baxter operator then the above conditions hold.
\mlabel{thm:gen}
\end{theorem}

\begin{proof}
Since $P$ is a monomial linear operator on $\bfk[x]$, the Rota-Baxter relation in Eq.~(\mref{eq:RB}) is equivalent to \begin{equation}
\beta(m)\beta(n)x^{\theta(m)+\theta(n)}
  =\beta(m+\theta(n))\beta(n)x^{\theta(m
  +\theta(n))}
  +\beta(\theta(m)+n)\beta(m)x^{\theta(
  \theta(m)+n)}, \quad \text{for all } m, n \in
  \NN.
\mlabel{eq:rbspe}
\end{equation}

Suppose (\ref{it:gen1}) and (\ref{it:gen2}) hold.
Since $\NN$ is the disjoint union of $\nul_\beta$ and $\supp_\beta$, we can verify Eq.~(\mref{eq:rbspe}) by considering the following four cases:
$$m, n\in \nul_\beta\,;\quad  m\in \nul_\beta, n\in \supp_\beta\,;\quad m\in \supp_\beta, n\in \nul_\beta\,;\quad m, n\in \supp_\beta\,.$$
In the first case we have $\beta(m)=\beta(n)=0$.  Thus Eq.~(\mref{eq:rbspe}) holds. In the second case, we have $\beta(m)=0$ and so Eq.~(\mref{eq:rbspe}) becomes $\beta(m+
\theta(n))\beta(n)=0$. Then Eq.~(\mref{eq:rbspe}) follows from Item~(\mref{it:gen1}). The third case can be treated similarly. In the last case, Eq.~(\mref{eq:rbspe}) follows from Eqs.~(\mref{eq:dege1}) and ~(\mref{eq:dege2}).  Thus $P$ is a Rota-Baxter operator on $\bfk[x]$.

Now assume that $\supp_\beta+\theta(\supp_\beta)\subseteq \supp_\beta$ and suppose that $P$ is a Rota-Baxter operator.  Then Eq.~(\mref{eq:rbspe}) holds. Taking $m\in \nul_\beta$ and $n\in \supp_\beta$, we obtain $0=\beta(m+\theta(n))\beta(n)x^{\theta(m+\theta(n))}.$ Since $\beta(n)\neq 0$, we must have $\beta(m+\theta(n))=0$, proving (\mref{it:gen1}). Taking $m, n\in \supp_\beta$, we have $\beta(m+\theta(n))\neq 0$ and $\beta(\theta(m)+n)\neq 0$ by the assumption. Then all the coefficients in Eq.~(\mref{eq:rbspe}) are nonzero. Thus the degrees of the
monomials must be the same; this yields Eq.~(\mref{eq:dege1}), and Eq.~(\mref{eq:dege2}) follows.
\end{proof}

By symmetry, only one of the two identities~(\mref{eq:dege1}) is needed. Note
also that by definition $A + \emptyset = \emptyset$ for any set $A$ so that
$\supp_\beta+\theta(\supp_\beta)\subseteq \supp_\beta$ and
$\nul_\beta+\theta(\supp_\beta)\subseteq \nul_\beta$ are automatic in the
nondegenerate case. Otherwise, we have the following constraint on
$\supp_\beta$.

\begin{lemma}If $P$ is a degenerate monomial Rota-Baxter operator on $\bfk[x]$,  then $\supp_\beta$ is either empty or infinite. The same applies to $\nul_\beta$.
\mlabel{lem:fin-or-cofin}
\end{lemma}
\begin{proof}
Suppose  $\supp_\beta\neq \emptyset$ and
$|\supp_\beta|=t<\infty$.  Then we may assume that
$$\supp_\beta=\{ m_i\in\NN\,|\, 1\leq i\leq t, m_1<\cdots<m_t\}.$$
By Eq.~(\mref{eq:rbspe}),
we have $\beta(m_t)^2=2\beta(m_t)\beta(m_t+
\theta(m_t))$.  Since $\beta(m_t)\neq 0$, we have $\beta(m_t)=2\beta(m_t+\theta(m_t))$, and so  $\beta(m_t+\theta(m_t))\neq 0$. Thus $m_t+\theta(m_t)$ is in $ \supp_\beta.$  By Lemma~\mref{lem:zero}, we have $\theta(m_t)\geq 1$. Then $m_t+\theta(m_t)>m_t$, a contradiction. Thus either $\supp_\beta=\emptyset$ or $|\supp_\beta|=\infty$.

On the other hand, let $\nul_\beta\neq \emptyset$. If $\nul_\beta=\NN$, then it is certainly infinite. If $\nul_\beta\neq \NN$, then take $k\in \supp_\beta$. Since $\supp_\theta = \supp_\beta$ by Lemma~\mref{lem:zero}, we have $\theta(k)>0$. By Theorem~\mref{thm:gen}(\mref{it:gen1}), we have $\nul_\beta+\theta(k)\subseteq \nul_\beta$. This implies that $\nul_\beta$ is infinite.
\end{proof}

We now give a general setup for constructing monomial Rota-Baxter operators on
$\bfk[x]$. This setup will be applied in Section~\mref{ss:nond} to construct
nondegenerate monomial Rota-Baxter operators and in Section~\mref{ss:dege} to
construct degenerate monomial Rota-Baxter operators.

\begin{theorem}
Let $\supp$ be a subset of \,$\NN$.
\begin{thmenumerate}
\item Let the maps $\theta:\supp\to \nonz{\NN}$ and $ \beta:\supp\to
  \nonz{\bfk}$ satisfy the following conditions.
\begin{thmenumerate}
\item We have
$\supp +\theta(\supp)\subseteq \supp$ and~$\NN\setminus \supp +\theta(\supp)\subseteq \NN\setminus \supp$.
\item
The equations~(\mref{eq:dege1}) and~(\mref{eq:dege2})
are fulfilled for all $m, n\in \supp$.
\end{thmenumerate}
Extend $\theta$ and $\beta$ to~$\NN$ by defining $\theta(n)=0$ and $\beta(n)=0$ for $n\in \NN\setminus \supp$.
Then $P:\bfk[x]\to\bfk[x]$ defined by $P(x^n)=\beta(n)x^{\theta(n)}, n\in \NN,$ is a Rota-Baxter operator on $\bfk[x]$.
\mlabel{it:betagen1}
\item
Let $\theta:\supp\to \nonz{\NN}$ satisfy Eq.~(\mref{eq:dege1}) and \,$\NN\setminus \supp +\theta(\supp)\subseteq \NN\setminus \supp$. Extend $\theta$ to $\NN$ by defining $\theta(n)=0$  for $n\in \NN\setminus \supp$. For any $c\in \nonz{\bfk}$,
define $\beta:\NN\to \bfk$ by
\begin{equation}
\beta(n)=\left\{\begin{array}{ll} c/\theta(n), & n\in \supp, \\
0, & n\not\in \supp. \end{array} \right .
\mlabel{eq:betagen}
\end{equation}
Then $P:\bfk[x]\to\bfk[x]$ defined by $P(x^n)=\beta(n)x^{\theta(n)}$ is a Rota-Baxter operator on $\bfk[x]$.
\mlabel{it:betagen2}
\end{thmenumerate}
\mlabel{thm:betagen}
\end{theorem}
\begin{proof}
(\mref{it:betagen1}) This follows from Theorem~\mref{thm:gen}.

(\mref{it:betagen2}) Under the assumption, we have for $m, n\in \supp$ that
\begin{align*}
& \beta(m+\theta(n))\beta(n) +\beta(\theta(m)+n)\beta(m)
= \frac{c^2}{\theta(m+\theta(n))\theta(n)} +\frac{c^2}{\theta(\theta(m)+n)\theta(m)}\\
&\quad= \frac{c^2}{(\theta(m)+\theta(n))\theta(n)} +\frac{c^2}{(\theta(m)+\theta(n))\theta(m)}\\
&\quad= \frac{c}{\theta(m)}\frac{c}{\theta(n)}\\
&\quad= \beta(m)\beta(n).
\end{align*}
Thus $\theta$ and~$\beta$ satisfy the conditions in Theorem~\mref{thm:gen} for $P$ to be a Rota-Baxter operator on $\bfk[x]$.
\end{proof}

\subsection{Nondegenrate case}
\mlabel{ss:nond}

As mentioned earlier, for a nondegenerate monomial linear operator $P$ on $\bfk[x]$, the
conditions $\supp_\beta+\theta(\supp_\beta)\subseteq \supp_\beta$ and
$\nul_\beta+\theta(\supp_\beta)\subseteq \nul_\beta$ are automatic.  Thus we
obtain the following characterization of nondegenerate monomial Rota-Baxter operators from Theorems~\mref{thm:gen} and
\mref{thm:betagen}.

\begin{coro}
\begin{thmenumerate}
\item Let~$P$ be a nondegenerate monomial linear operator on~$\bfk[x]$
  as in Eq.~(\mref{eq:mono}). Then~$P$ is a Rota-Baxter operator if and only if the
  sequences~$\theta$ and~$\beta$ satisfy the equations~(\mref{eq:dege1}) and~(\mref{eq:dege2})
  for all~$m, n \in \NN$. In this case, $\theta(n)\neq 0$ for all $n
  \in \NN$.  \mlabel{it:bt1}
\item If a sequence~$\theta\colon \NN\to \NN$ is nonzero and
  satisfies Eq.~(\mref{eq:dege1}), then for any $c\in \nonz{\bfk}$,
 the map
$\beta\colon \NN\to \bfk$ given by $\beta(n):=c/\theta(n)$
  satisfies Eq.~(\mref{eq:dege2}) and hence gives a Rota-Baxter operator on
  $\bfk[x]$.  \mlabel{it:bt2}
\end{thmenumerate}
\mlabel{coro:bt}
\end{coro}

Equation~(\mref{eq:dege1}) characterizes~$\theta$ as an
averaging operator defined as follows.

\begin{defn}
\begin{thmenumerate}
\item
A map $\theta: S\to S$ on a semigroup $S$ is called an {\bf averaging operator} if
$$ \theta(m \theta(n))=\theta(m)\theta(n) \quad \text{ for all } m, n\in S.$$
\item
A linear map $\Theta:R\to R$ on a $\bfk$-algebra $R$ is called an {\bf averaging operator} if $\Theta$ is an averaging operator on the multiplicative semigroup of $R$.
\end{thmenumerate}
\end{defn}

The study of averaging operators can be tracked back to Reynolds and Birkhoff~\cite{Bi,Re}.
 We refer the reader to \cite{GP} and the references therein for further details.

By Corollary~\mref{coro:bt}, a nondegenerate monomial operator $P$ on $\bfk[x]$ is a Rota-Baxter operator if and only if the map $\theta$ is an averaging operator on the semigroup~$(\NN, +)$, and the
corresponding $\bfk$-linear operator~$\Theta\colon x^n \mapsto
x^{\theta(n)}$ makes~$(\bfk[x], \Theta)$ into an averaging algebra
We write~$\mathcal{A}$ for the set of all
nondegenerate averaging operators, i.e.\@ sequences~$\theta\colon \NN
\to \nonz{\NN}$ satisfying Eq.~(\mref{eq:dege1}).
We describe~$\mathcal{A}$ as the first step to determine nondegenerate monomial Rota-Baxter operators on $\bfk[x]$. We denote the free semigroup
over~$\nonz{\NN}$ by~$S(\nonz{\NN})$, so the elements~$\sigma \in
S(\nonz{\NN})$ are \emph{finite sequences}~$(\sigma_0, \cdots,
\sigma_{d-1})$ of positive numbers having any length~$d > 0$.

\begin{theorem}
  \label{thm:theta-char}
  There is a bijective correspondence~$\Phi\colon \mathcal{A} \to
  S(\nonz{\NN})$ given by
  \begin{equation*}
    \Phi(\theta) = \Big(\theta(0), \cdots, \theta(d-1)\Big)\,\Big/\,d
    \quad\text{with}\quad
    d := \min \, \{j \in \nonz{\NN} \mid \theta(r+j) = \theta(r) + j
    \;\text{for all $r \in \NN$}\}
  \end{equation*}
  whose inverse maps~$\sigma:=(\sigma_0, \cdots, \sigma_{d-1}) \in S(\nonz{\NN})$ to the
  map~$\theta\colon \NN \to \nonz{\NN}$ defined by~$\theta(n) = (\ell+\sigma_j)\,d$
  for~$n = \ell d + j$ with~$\ell \in \NN$ and~$0 \le j < d$. Moreover, we
  have~$\im(\theta) = d \NN_{\ge s}$ for~$s := \min(\sigma)$.
\end{theorem}
\begin{proof}
First consider~$\theta \in \mathcal{A}$. Defining the
  map~$\tilde{\theta} := \theta - \id_{\NN}\colon \NN \to \ZZ$, one
  obtains from Eq.~(\mref{eq:dege1}) that~$\tilde{\theta}(m+\theta(n))
  = \tilde{\theta}(m)$ for all~$m,n \in \NN$. Hence~$\tilde{\theta}$
  is periodic, and~$d$ is well-defined as the primitive period
  of~$\tilde{\theta}$. Since every~$\theta(n)$ is also a period
  of~$\tilde{\theta}$, this implies~$\im(\theta) \subseteq d
  \nonz{\NN}$ so that the given map~$\Phi\colon \mathcal{A} \to
  S(\nonz{\NN})$ is well-defined.

  Next let us write~$\Psi$ for the assignment~$\sigma \mapsto \theta$
  defined above. By checking Eq.~(\mref{eq:dege1}) one sees that this
  yields a well-defined map~$\Psi\colon S(\nonz{\NN}) \to
  \mathcal{A}$.

 Now we prove~$\Phi \circ \Psi = \id_{S(\nonz{\NN})}$, so
  let~$\theta\colon \NN \to \nonz{\NN}$ be the map defined as above by
  a given sequence~$(\sigma_0, \cdots, \sigma_{d-1}) \in
  S(\nonz{\NN})$. Since~$\tilde{\theta}(n) = \sigma_j d - j$ for~$n =
  \ell d+j$, we see that~$d$ is a period of the
  map~$\tilde{\theta}$. Assume~$d$ is greater than its primitive
  period~$d'$. Then we have~$d = kd'$ for~$k > 1$, and
  \begin{equation*}
    \sigma_0 kd' = \theta(0) = \tilde{\theta}(0) = \tilde{\theta}(d')
    = \sigma_{d'} d - d' = (\sigma_{d'} k  - 1) \, d'
  \end{equation*}
  implies~$\sigma_0 k = \sigma_{d'} k  - 1$, which contradicts~$k >
  1$. We conclude that~$d$ is the primitive period
  of~$\tilde{\theta}$, so the definition of~$\Phi$ recovers the
  correct value of~$d$. Moreover, for~$j=0, \cdots, d-1$ we
  have~$\theta(j) = \sigma_j d$, which implies~$\Phi(\theta) = \sigma$
  as required.

  It remains to prove the converse relation~$\Psi \circ \Phi =
  \id_{\mathcal{A}}$. Taking an arbitrary~$\theta \in \mathcal{A}$, we must
  prove that it coincides with the sequence~$\theta'$ defined by~$\theta'(\ell
  d+j) = (\ell+\theta(j)/d) \, d = \ell d+\theta(j)$ for any~$\ell \in \NN$
  and~$0 \le j < d$. For these values we must then show that~$\theta(\ell d+j) =
  \ell d + \theta(j)$, which is equivalent to~$\tilde{\theta}(\ell d+j) =
  \tilde{\theta}(j)$. The latter is ensured since we know that~$\tilde{\theta}$
  has primitive period~$d$.

  As noted above, $\im(\theta) \subseteq d \nonz{\NN}$ so~$\theta/d\colon \NN
  \to \nonz{\NN}$ is well-defined. We must show~$\im(\theta/d) = \NN_{\ge
    s}$. The inclusion from left to right follows since~$(\theta/d)(\ell d+j) =
  \ell + \sigma_j \ge \sigma_j \ge s$. Now let~$n \ge s$ be given and
  write~$s=\sigma_j$ for some~$j = 0, \cdots, d-1$. Then~$\ell := n - \sigma_j
  \in \NN$ is such that~$(\theta/d)(\ell d+j)=n$, which established the
  inclusion from right to left.
\end{proof}

As sequences, the relation between~$\theta\colon \NN \to \nonz{\NN}$
and~$\sigma\colon \{0, \cdots, d-1\} \to \nonz{\NN}$ can be written as~$\theta/d
= (\sigma, \sigma+1, \sigma+2, \cdots)$, where~$1, 2, \cdots$ designate constant
sequences of length~$d$. More precisely, we have
$$ \theta/d=(\sigma_0,\cdots,\sigma_{d-1},\sigma_0+1,\cdots,\sigma_{d-1}+1,\cdots).$$

Theorem~\mref{thm:theta-char} yields the following construction algorithm for the map $\theta$ from a
nondegenerate monomial Rota-Baxter operator.
\begin{algorithm}
Every sequence~$\theta\colon
\NN\to \NN$ corresponding to a nondegenerate monomial Rota-Baxter operator on~$\bfk[x]$ can be generated as follows:
\begin{enumerate}
\item Let $d\in \nonz{\NN}$ be given. For each $j=0 \dots d-1$ fix
  $\sigma_j\in\nonz{\NN}$;
\item For~$n\in \NN$ with $n=\ell d + \md{n}$ where $\md{n}\in\{0 \dots d-1\}$
  is the remainder of $n$ modulo $d$, define
$$\theta(n) := n+\sigma_{\md{n}}\,d-\md{n}=\ell d +\sigma_{\md{n}}\,d.$$
\end{enumerate}
\mlabel{al:theta}
\end{algorithm}

\noindent We consider two \emph{extreme cases} of Algorithm~\mref{al:theta} of particular interest:
\begin{description}
\item[Case 1] If~$d=1$ one can only choose $\theta(0)\neq 0$ so that~$\theta(n)=n+\theta(0)$ for all $n \in \NN$.
\item[Case 2] For~$d>1$ and $\sigma_j=1, 0\leq j\leq d-1$ we have ~$\theta(n) = n + d - \md{n}=(\ell+1)d$
  with~$n=\ell d+\md{n}$.
\end{description}

\begin{exam}
  Setting~$d=2$ and $\sigma_0=\sigma_1=1$, we choose the sequence~$\beta$
  according to Corollary~\mref{coro:bt}(\mref{it:bt2}) with~$c=2$.  Then the
  $\bfk$-linear map $P\colon \bfk[x]\rightarrow \bfk[x]$ by
$$
P (x^{2k})= \frac{ x^{2k+2}}{k+1}\quad \text{and} \quad P(x^{2k+1})=
\frac{x^{2k+2}}{k+1}$$ is a nondegenerate  Rota-Baxter
operator on~$\bfk[x]$.
\end{exam}

We determine next all~$\beta$ for the sequences~$\theta$ coming from the above
two extreme cases.

\begin{theorem}
\begin{thmenumerate}
\item Let $d=1$ with $\theta(n)=n+k$ for some $k\in\nonz{\NN}$.  Then $\beta\colon \NN \to \bfk$ satisfies Eq.~(\mref{eq:dege2}) if
  and only if $\beta(n)=\frac{c}{\theta(n)}$ for some $c\in\nonz{\bfk}$.
  \mlabel{it:fircase}
\item Let $d>1$ be given with $\theta(n)=n+d-\md{n}$. Then $\beta\colon \NN \to
  \bfk$ satisfies Eq.~(\mref{eq:dege2}) if and only if it is defined
  as follows: Fix $c_j\in \nonz{\bfk}$ and assign $\beta(j):=1/c_j$ for $0\leq
  j\leq d-1$.  Then for any $n\in \NN$ with $n=\ell d+\md{n}$ define
  $\beta(n)=\frac{\beta(\md{n})}{\ell+1}.$ \mlabel{it:seccase}
\end{thmenumerate}
\mlabel{thm:gbeta}
\end{theorem}

\begin{proof}
(\mref{it:fircase}) For a
$\theta$ of the given form, by Eq.~(\mref{eq:dege2}), we have
\begin{equation}
\beta(n)\beta(0)=\beta(n+k)(\beta(0)+\beta
(n)).
\mlabel{eq:injnon}
\end{equation}
Set $\beta(0):=a$ for some $a\in\nonz{\bfk}$ and write $c:=ka$. Then  $\beta(0)=\frac{c}{k}$ and $c$ is in $\nonz{\bfk}$. We next prove $\beta(n)=\frac{c}{n+k}$ by induction on $n\geq 0$. The base case $n=0$ is true. Assume $\beta(n)=\frac{c}{n+k}$ has been proved for $n\geq 0$. By Eq.~(\mref{eq:injnon}), we obtain
\begin{equation}
\beta(n+1-k)\beta(0)=\beta(n+1)(\beta(
0)+\beta(n+1-k)).
\mlabel{eq:indu}
\end{equation}
Since $k\geq 1$, we have $n+1-k\leq n$. By the induction hypothesis, we get $\beta(n+1-k)=\frac{c}{n+1}$.
Then by Eq.~(\mref{eq:indu}) we have
$$\beta(n+1)=\frac{\frac{c^2}{k(n+1)}}{
\frac{c}{k}+\frac{c}{n+1}}=\frac{c}{n+1+k}.
$$
This completes the induction. Thus $\beta(n)=\frac{c}{\theta(n)}$  for some $c\in\nonz{\bfk}$ and all $n\in\NN$.

\noindent The converse follows from Theorem~\mref{coro:bt}(\mref{it:bt2}).
\smallskip

\noindent
(\mref{it:seccase})
Taking $\gamma(n)=1/\beta(n)$, Eq.~(\mref{eq:dege2}) is equivalent to
  \begin{equation}
    \label{eq:gamma}
    \frac{\gamma(m)}{\gamma(m+\theta(n))} + \frac{\gamma(n)}{\gamma(\theta(m)+n)}
    = 1.
  \end{equation}
  Thus we just need to show that, for a fixed sequence $\theta$ in the theorem,
  a sequence $\gamma:\NN\to \bfk$ satisfies Eq.~(\mref{eq:gamma}) if and only if
  $\gamma$ is defined by $\gamma (n)=(\ell +1) \gamma (\md{n})$ if $n=\ell d +
  \md{n},$ where the $\gamma(\md{n})\in \nonz{\bfk}$ for~$\md{n} \in \{0 \dots
  d-1\}$ are arbitrarily preassigned.

\noindent
$(\Longrightarrow)$
Take $m=0$ and $n=\ell d$ with $\ell\geq 0$ in Eq.~(\mref{eq:gamma}). After simplifying we obtain
$$ \gamma ((\ell +1)d)=\gamma (\ell d)+\gamma (0).$$
Then by an induction on $\ell$, we obtain
\begin{equation}
\gamma(\ell d)=(\ell+1)\gamma(0).
\mlabel{eq:md}
\end{equation}
Next note that for $n=\ell d +\md{n}$ we have
\begin{equation}
\theta (n)=\ell d + d.
\mlabel{eq:thetan}
\end{equation}
Then for $j\in \{0,\cdots,d-1\}$, taking $m=0$ and $n=\ell d+j$ in Eq.~(\mref{eq:gamma}) we obtain

$$1= \frac{\gamma(0)}{\gamma(\theta(\ell d+j ))}+\frac{\gamma(\ell d+j)}{\gamma(\theta(0)+ \ell d+j )}
= \frac{\gamma(0)}{\gamma(\ell d +d)} + \frac{\gamma(\ell d+j)}{\gamma (d+\ell d+j)}\,.
$$
This gives
$$ \gamma ((\ell +1)d+j)= \frac{\ell+2}{\ell+1} \gamma (\ell d+j)$$
and recursively yields
$$ \gamma (\ell d+j)=(\ell+1)\gamma (j).$$

\smallskip

\noindent
$(\Longleftarrow)$ Conversely, suppose a sequence $\beta$ is given by $\gamma
(n)=(\ell +1) \gamma (\md{n})$ if $n=\ell d + \md{n},$ for preassigned
$\gamma(\md{n})$ as specified above.  Then for any $m, n\in \NN$ with
$m=kd+\md{m}$ and $n=\ell d+\md{n}$, by Eq.~(\mref{eq:thetan}) we obtain
\begin{eqnarray*}
    \frac{\gamma(m)}{\gamma(m+\theta(n))} + \frac{\gamma(n)}{\gamma(\theta(m)+n)}
&=&
\frac{\gamma(kd+\md{m})}{\gamma(kd+\md{m} +\theta(\ell d +\md{n}))} + \frac{\gamma(\ell d+\md{n})}{\gamma(\theta(kd+\md{m}) +\ell d+\md{n}) }\\
&=&
\frac{\gamma(kd+\md{m})}{\gamma(kd+\md{m} +\ell d +d)} + \frac{\gamma(\ell d+\md{n})}{\gamma(kd+d +\ell d+\md{n}) }\\
&=&
\frac{(k+1)\gamma(\md{m})}{(k+\ell +2)\gamma(\md{m})} + \frac{(\ell +1)\gamma(\md{n})}{(k+\ell+2)\gamma(\md{n})} = 1.
\end{eqnarray*}
This is Eq.~(\mref{eq:gamma}).
\end{proof}

In the special case of \emph{polynomial sequences}~$\theta\colon \NN
\to \NN$ and~$\alpha = 1/\beta\colon \NN \to \bfk$, the range of
possibilities can be drastically narrowed down.

\begin{theorem}
  Suppose $\bfk$ is a field containing $\QQ$.  Let $P\colon
  \bfk[x]\rightarrow \bfk[x]$ be a nondegenerate monomial linear
  operator with $P(x^n)=\frac{1}{\alpha(n)}x^{\theta(n)}$ for
  $n\in\NN$, and assume $\theta(n)$ as well as $\alpha(n)$ are
  polynomials. Then $P$ is a Rota-Baxter operator if and only if
\begin{equation}
  \theta(n)=n+k\qquad\text{and}
  \qquad\alpha(n)=c(n+k)
  \mlabel{eq:ta}
\end{equation}
for some~$k \in \NNP$ and some~$c\in\nonz{\bfk}$.
\label{thm:polycase}
\end{theorem}
\begin{proof}
  By Corollary~\mref{coro:bt}, the operator~$P$ defined by Eq.~(\mref{eq:ta})
  is a Rota-Baxter operator. So we just need to show that any
  Rota-Baxter operator given by Eq.~(\mref{eq:mono}) with polynomial
  sequences~$\theta(n)$ and~$\alpha(n)$ must satisfy the
  conditions in Eq.~(\mref{eq:ta}). Since~$P$ is a Rota-Baxter operator,
  Eq.~(\mref{eq:dege1}) gives the characteristic
  relation~$2\theta(n)=\theta(\theta(n)+n)$. But~$\theta$ and~$\alpha$
  are polynomials with $\deg \theta$ and $\deg \alpha$
  respectively. Checking degrees, let us first assume~$\deg\theta \geq
  2$. In this case we have
  $$\deg 2\theta=\deg \theta<
  (\deg\theta)^2=\deg\theta(\theta(n)+n),$$ which contradicts the
  characteristic relation. Thus we have $\deg\theta\leq 1$, and we can
  write $\theta(n)=sn+k$ for some~$s,k\in\NN$. Now the characteristic
  relation becomes~$2(sn+k) = s(sn+n+k)+k$ or
  equivalently~$(sn+k)(s-1)=0$. If~$s\neq 1$ we obtain~$sn+k=0$ for
  all~$n \in \NN$. But then $s=k=0$, and~$P$ is the zero operator,
  which contradicts the hypothesis that~$P$ is
  nondegenerate. Therefore~$s=1$ and hence~$\theta(n)=n+k$ as claimed
  in Eq.~(\mref{eq:ta}).

  For deriving the second condition of Eq.~(\mref{eq:ta}), we
  specialize Eq.~(\mref{eq:dege2}) to obtain $2\alpha(n)=
  \alpha(\theta(n)+n)$ and hence the
  recursion~$2\alpha(n)=\alpha(2n+k)$. Set~$\ell = \deg\alpha$ and
  suppose the leading coefficient of~$\alpha$
  is~$c\in\nonz{\bfk}$. Now taking leading coefficients of the
  recursion, we get~$2c=2^{\ell}c$ and thus~$\ell=1$. This means we
  can write~$\alpha(n)=c n+c_0$ for some~$c \in
  \nonz{\bfk}$ and~$c_0 \in\bfk$. Substituting this into the
  recursion leads to $2(c n+c_0)=c(2n+k)+c_0$ and
  hence~$\alpha(n)=c(n+k)$ as claimed in Eq.~(\mref{eq:ta}). It
  remains to show that~$k \ne 0$. But this follows
  because~$P(1)=x^k/c k$ so that necessarily~$c k\neq0$.
\end{proof}

Next we investigate injective monomial Rota-Baxter operators and show them to be
analytically modelled. We note first that if $P$ is degenerate, then there
exists $n_0\in\NN$ such that $\beta(n_0)=0$, and then $P(x^{n_0})=0$. Thus
$\ker(P)\neq \{0\}$ and $P$ is not injective. Thus any injective monomial
Rota-Baxter operator is nondegenerate.

\begin{theorem}
Let~$P$ be a monomial Rota-Baxter operator on~$\bfk[x]$. The
following statements are equivalent.
\begin{thmenumerate}
\item
The operator $P$ is injective. \mlabel{it:inj}
\item
The $\theta$ as in Eq.~(\mref{eq:mono}) from $P$ satisfies
$\theta(n)=n+k$ for some $k\in\nonz{\NN}$.
\mlabel{it:trans}
\item There are $k\in\nonz{\NN}$ and $c\in \nonz{\bfk}$ such that $P(x^n)=
  c\int_0^xt^{n+k-1}dt$ and hence $P=c\stdint{0}x^{k-1}.$ \mlabel{it:int}
\end{thmenumerate}
\mlabel{thm:injint}
\end{theorem}
\begin{proof}
  (\mref{it:inj}) $\Rightarrow$ (\mref{it:trans}): Assume that $P$ is an
  injective monomial Rota-Baxter operator. Then $P$ is nondegenerate. By
  Algorithm~\mref{al:theta}, there are $d\geq 1$ and $\sigma_j\in \nonz{\NN}$
  for $j\in \{0,\cdots,d-1\}$ such that $\theta(n)=\ell d +\sigma_{\md{n}} d$
  where~$n = \ell d + \md{n}$ and~$\md{n}$ is the remainder of $n$ modulo
  $d$. Suppose $d>1$. Without loss of generality, we may assume $\sigma_0\geq
  \sigma_1$ so that~$n:=(\sigma_0-\sigma_1)d+1 > 0$. Since $\theta(0)=\sigma_0
  d$ we have
$$\theta(n)=\theta((\sigma_0-\sigma_1)d+1)=(\sigma_0-\sigma_1)d+\sigma_1 d=\theta(0),$$
hence $\theta$ is not injective. This forces $d=1$. Then
by the first case considered after Algorithm~\mref{al:theta}, we
have $\theta(n)=n+k$ for fixed $k\geq 1$.
\smallskip

\noindent (\mref{it:trans}) $\Rightarrow$ (\mref{it:int}): For a
$\theta$ of the given form, by Theorem~\mref{thm:gbeta}(\mref{it:fircase}), we have
$\beta(n)=c/\theta(n)$ for some $c\in\nonz{\bfk}$. Thus
$$P(x^n)=\beta(n)x^{\theta(n)}=\frac{c}{n+k}
x^{n+k}=c\int_0^xt^{n+k-1}dt,$$ as needed.
\smallskip

\noindent (\mref{it:int}) $\Rightarrow$ (\mref{it:inj}): Since
$P(x^n)=c\int_0^xt^{n+k-1}dt = \frac{c}{n+k} x^{n+k}$ for all
$n\in\NN$, the operator $P$ is injective.
\end{proof}

\subsection{Degenerate case}
\mlabel{ss:dege}
We next apply Theorem~\mref{thm:betagen} to construct degenerate monomial Rota-Baxter operators on $\bfk[x]$ when $\supp_\beta$ is either $k\NN$ where $k\geq 1$ or is $\NN\setminus (k\NN)$ where $k\geq 2$.

\begin{prop}
  Let $P(x^n)=\beta(n)x^{\theta(n)}, n\in \NN,$ define a monomial linear
  operator on $\bfk[x]$ such that $\supp_\beta=k\NN$ for some $k>0$.  Then $P$
  is a Rota-Baxter operator on $\bfk[x]$ if and only
  if~$\theta(km)=\tilde{\theta}(m)\in \nonz{ \supp_\beta}$ and
  $\beta(km)=\tilde{\beta}(m), m\geq 0$ for maps $\tilde{\theta}:\NN\to \NN$ and
  $\tilde{\beta}:\NN\to \bfk$ that satisfy the following equations
  \begin{align}
&\tilde{\theta}(m_1)+\tilde{\theta}(m_2) =\tilde{\theta}(m_1+\frac{1}{k}\tilde{\theta}(m_2))
=\tilde{\theta}(\frac{1}{k}\tilde{\theta}
(m_1)+m_2),
\mlabel{eq:tildb1}\\
&\tilde{\beta}(m_1)\tilde{\beta}(m_2)
=\tilde{\beta}(m_1
+\frac{1}{k}\tilde{\theta}(m_2))\tilde{\beta}(m_2)+\tilde{\beta}(
m_2+\frac{1}{k}\tilde{\theta}(m_1))\tilde{\beta}(m_1)
\quad\text{ for all } m_1,m_2\in\NN.
\mlabel{eq:tildb2}
\end{align}

 \mlabel{pp:monog}
\end{prop}

\begin{proof}
Since $\supp_\beta=\{km\,|\,m\in\NN\}$, we have $\nul_\beta=\{km+i\,|\,1\leq i\leq k-1, m\in\NN\}$.
Suppose $P$ is a Rota-Baxter operator on $\bfk[x]$. Then by Eq.~(\mref{eq:RB}), we have $P(x^{km_1+i}P(x^{km_2}))=0$ for all $m_1,m_2\in\NN$ and $1\leq i\leq k-1$.
Thus $\beta(km_2)\beta(km_1+i+\theta(km_2))=0$.
Since $\beta(km_2)\neq 0$, we have $\beta(km_1+i+\theta(km_2))=0$. Then $km_1+i+\theta(km_2)$ is in $\nul_\beta,$ and then $i+\theta(km_2)$ is in $\nul_\beta$ for $1\leq i\leq k-1.$ Suppose that there exists $m_0\in\NN$ such that $\theta(km_0)\nequiv 0$ $(\bmod\, k).$ Then there exists  $1\leq i_0\leq k-1$ such that $i_0+\theta(km_0)\equiv0\, (\bmod\,k)$. So $i_0+\theta(km_0)$ is in $ \supp_\beta$ by the definition of $\supp_\beta$. This is a contradiction to the fact proved above that $i+\theta(km_2)$ is in $\nul_\beta$ for $1\leq i\leq k-1$.

Thus $\theta(km)$ is in $\nonz{\supp_\beta}$ for all $m\in\NN$.
So $kn+\theta(km)$ is in $ \supp_\beta$ for all $n,m\in \NN$. By Theorem~\mref{thm:gen}, Eqs.~(\mref{eq:dege1}) and ~(\mref{eq:dege2}) hold. Let $\tilde{\theta}(m):=\theta(km)$ and let $\tilde{\beta}(m):=\beta(km)$, $m\in\NN.$ Thus $\tilde{\beta}(m)\neq 0$ for all $m\in \NN$.
Then by Eqs.~(\mref{eq:dege1}) and ~(\mref{eq:dege2}), Eqs.~(\mref{eq:tildb1}) and ~(\mref{eq:tildb2}) hold.
This is what we want. The converse follows from Theorem~\mref{thm:betagen}(\mref{it:betagen1}).
\end{proof}

Proposition~\mref{pp:monog} gives a large class of monomial Rota-Baxter operators on $\bfk[x]$ with $\supp_\beta=k\NN$, reducing to Corollary~\mref{coro:bt} for~$k=1$. On the other hand, Theorem~\mref{thm:betagen} also gives the following result on monomial Rota-Baxter operators on $\bfk[x]$ where $\supp_\beta$ is now complementary to Proposition~\mref{pp:monog}.

\begin{prop}
  Let $P(x^n)=\beta(n)x^{\theta(n)}$ be a monomial linear operator on $\bfk[x]$
  with $\supp_\beta=\NN \setminus k\NN$ for some $k\geq 2$.
\begin{thmenumerate}
\item For any~$t\in\NNP$ one obtains a degenerate monomial RBO by
  setting~$\theta(km+i)=k(m+t)$ and~$\theta(km)=0$ for $m\in \NN$ and $1\leq
  i\leq k-1$, choosing~$\beta$ as in
  Theorem~\mref{thm:betagen}(\mref{it:betagen2}).  \mlabel{it:dege21}
\item Assume that $\theta(i)=k$ for $1\leq i\leq k-1$.  Then $\theta$
  corresponds to a degenerate monomial RBO on $\bfk[x]$ if and only if
  $\theta(km+i)=k(m+1)$ for all $m\in\NN$ and $1 \le i \le k-1$.
  \mlabel{it:dege22}
\end{thmenumerate}
\mlabel{pp:dege2}
\end{prop}
\begin{proof}(\ref{it:dege21}) By our assumption on $\supp_\beta$, we have
  $\nul_\beta=\{km\,|\,m\in\NN\}$. By assumption $\theta(km+i)=k(m+t)$ for all
  $m\in\NN$ and $1\leq i\leq k-1$, hence we obtain
  $\nul_\beta+\theta(\supp_\beta)\subseteq \nul_\beta$. Since
$$\theta(km_1+i_1)+\theta(km_2+i_2)=k(m_1+
m_2+2t)$$
and
$$\theta(km_1+i_1+\theta(km_2+i_2))=\theta
(k(m_1+m_2+t)+i_1)=k(m_1+m_2+2t)$$ for all $m_1,m_2\in\NN$ and $1\leq
i_1,i_2\leq k-1$, we have Eq.~(\mref{eq:dege1}). Thus we may apply
Theorem~\mref{thm:betagen}(\mref{it:betagen2}) to obtain a degenerate RBO $P$ on
$\bfk[x]$.  \smallskip

\noindent
(\mref{it:dege22}) Assume first that $P$ is a monomial RBO on $\bfk[x]$. Then by
Eq.~(\mref{eq:rbspe}), $\beta(km+i)\beta(km+\theta(km+i))=0$ for all $m\in\NN$
and $1\leq i\leq k-1$. Since $\beta(km+i)\neq 0$, we have
$\beta(km+\theta(km+i))=0$, so $km+\theta(km+i)$ is in $\nul_\beta$.  From
$\nul_\beta=k\NN$ we infer $\theta(km+i)\in\nul_\beta$. Thus
$\supp_\beta+\theta(\supp_\beta)\subseteq \supp_\beta$. By
Theorem~\mref{thm:gen}, Eq.~(\mref{eq:dege1}) holds. We now prove that
$\theta(km+i)=k(m+1)$ by induction on $m\geq 0$. The base case $m=0$ is
immediate from our assumption. Assume $\theta(km+i)=k(m+1)$ has been proved for
$m\geq 0$. By Eq.~(\mref{eq:dege1}), we have
$$\theta(k(m+1)+i)=\theta(km+i+\theta(i))=\theta(
km+i)+\theta(i).$$
By the induction hypothesis, we get $\theta(k(m+1)+i)=k(m+2)$. This completes the proof.

Conversely, by $\theta(km+i)=k(m+1)$ and Item~(\mref{it:dege21}), we obtain a
degenerate RBO $P$ on $\bfk[x]$.
\end{proof}

\begin{exam}
Taking $k= 2$ in Proposition~\mref{pp:monog} and Proposition~\mref{pp:dege2}, we obtain the following degenerate monomial Rota-Baxter operators on $\bfk[x]$.
\begin{enumerate}
\item
$ P(x^{2k})=\frac{x^{2(k+1)}}{k+1} \text{ and } P(x^{2k+1})=0 \text{ for all } k\in \NN.
$
\mlabel{it:degm2}
\item
$P(x^{2k})=0  \text{ and } P(x^{
2k+1})=\frac{x^{2(k+1)}}{k+1} \text{for all}\, k\in\NN.$
\mlabel{it:degm1}
\end{enumerate}
\mlabel{exam:deg}
\end{exam}

The above examples may also be regarded as special cases of the following
result.
\begin{prop}
  Let $P_0\in \rbo(R)$ for a $\bfk$-algebra~$R$. Assume~$\phi$ is a $\bfk$-linear
  operator on~$R$ such that $E := P_0(\im(\phi))$ is a nonunitary
  $\bfk$-subalgebra. If~$\phi$ is a homomorphism of the $E$-module~$R$ then
  $P_0 \circ \phi$ is also a Rota-Baxter operator on~$R$.  \mlabel{prop:newrbo}
\end{prop}
\begin{proof}
  This follows immediately since
  \begin{eqnarray*}
    (P_0\circ \phi)(a)(P_0\circ\phi)(b)&=&
    P_0(\phi(a))P_0(\phi(b))=
    P_0(\phi(a)P_0(\phi(b)))+P_0(P_0(\phi(a))\phi(b))\\
    &=&(P_0\circ\phi)(a(P_0\circ\phi)(b))+(P_0\circ\phi)
    ((P_0\circ\phi)(a)b),
  \end{eqnarray*}
  for all $a,b\in R$.
\end{proof}

For $R = \bfk[x]$ let $\phi\colon f(x)\mapsto (f(x)+f(-x))/2$ be the projector
onto the $\bfk$-subspace spanned by the even monomials and set
${P_0}=2\stdint{0}x$. Then
$$({P_0}\circ \phi)(x^n):=
\begin{cases}
\frac{x^{2(k+1)}}{k+1} & \text{if $n=2k$,}\\
0 & \text{if $n=2k+1$,}
\end{cases}$$
for all $k\in\NN$ so that ${P_0}\circ \phi$ is the same as $P$ in Example ~\mref{exam:deg}(\mref{it:degm2}). On the other hand, choosing~$\phi$ as the projector $f(x) \mapsto (f(x)-f(-x))/2$ onto the space of odd
monomials and setting ${P_0} = 2 \stdint{0}$ yields
$$({P_0}\circ \phi)(x^n):=
\begin{cases}
0 & \text{if $n=2k$,}\\
\frac{x^{2(k+1)}}{k+1} & \text{if $n=2k+1$,}
\end{cases}$$
for all $k\in\NN$ so that ${P_0}\circ \phi$ is the same as $P$ in Example ~\mref{exam:deg}(\mref{it:degm1}). In both cases, $E$ is the nonunitary algebra of nonconstant even monomials.

\begin{prop} Let $P(x^n)=\beta(n)x^{\theta(n)}$ be a nonzero degenerate monomial
RBO on $\bfk[x]$ satisfying the condition
$\supp_\beta+\theta(\supp_\beta)\subseteq \supp_\beta$.
\begin{thmenumerate}
\item
There exists a map~$\sigma\colon \NN \to
  \supp_\beta$ such that~$P_0(x^n) := P(x^{\sigma(n)})$ defines a nondegenerate
  monomial RBO on $\bfk[x]$.
\mlabel{it:degcls1}
\item
We have
  \begin{equation}
    \label{eq:beta-shape} \supp_\beta = C \uplus (s_1 + e \NN) \uplus \cdots
\uplus (s_k + e \NN),
  \end{equation}
  where~$C \subset \NN$ is finite, $k < e \in \nonz{\NN}$, and~$s_1, \dots, s_k
\in \supp_\beta$ are incongruent modulo $e$ $($in the sense that $x-y \not\in e \ZZ$$)$ such that~$s_1 - e, \dots, s_k
- e \not\in C$. Moreover, there exists a finite set~$E \subset \supp_\beta$ such
that~$\theta$ is determined uniquely by its values on~$E$.
\label{it:degcls2}
\end{thmenumerate}
\label{pp:degcls}
\end{prop}
\begin{proof} Since~$P$ is nonzero, both~$\supp_\beta \neq \emptyset$
and~$\nul_\beta \neq \emptyset$ are infinite by
Lemma~\mref{lem:fin-or-cofin}. From Eq.~(\mref{eq:dege1}) and the condition
$\supp_\beta+\theta(\supp_\beta)\subseteq \supp_\beta$ we see that~$T :=
\theta(\supp_\beta)$ is additively closed. As in the proof of
Theorem~\mref{thm:theta-char} one checks that~$\theta-\id_\NN$ is periodic
on~$\supp_\beta$ with primitive period~$d$ and~$T \subseteq d \nonz{\NN}$ so
that~$d \mid e := \gcd(T)$. Hence~$T/e$ is a numerical
semigroup~\cite[Prop.~10.1]{RosGar}, meaning a subsemigroup of~$\nonz{\NN}$ with
a finite complement~$G \subseteq \nonz{\NN}$ of so-called gaps. Thus we
obtain~$T = e \nonz{\NN} \setminus eG$.  We write~$f \in \NN$ for the Frobenius
number of~$T/e$, meaning the greatest element of~$G$ for~$G \neq \emptyset$
and~$f=0$ otherwise.
\smallskip

\noindent
(\ref{it:degcls1})
  Fix an element $s$ of~$\supp_\beta$. We define~$\sigma\colon \NN
  \to \supp_\beta$ as follows. For $n\in \NN$, write $n=\ell e+r$ with $\ell\geq 0$ and $0\leq r<e$. Define $\sigma(\ell e + r) :=
  (f+\ell)e + s$. Then~$\sigma\colon
  \NN \to \supp_\beta$ follows from the condition~$\supp_\beta+ T \subseteq
  \supp_\beta$ since~$(f+\ell)e \in T$ for all~$\ell > 0$. We show now that
  \begin{equation}
    \sigma\Big(n+\theta(\sigma(m))\Big) = \sigma(n) + \theta(\sigma(m))
    \mlabel{eq:sigma}
  \end{equation}
  for all~$m, n \in \NN$. We have~$\theta(\sigma(m)) = te \in T$ for some~$t
  \not\in G$, and we may write~$n = \ell e + r$ with~$0 \le r < e$ and~$\ell \ge
  0$. Then one computes~$\sigma(r) + (f+\ell+t)e$ for both sides
  of Eq.~(\mref{eq:sigma}).

  Let us now prove that~$P_0$ satisfies Eq.~(\mref{eq:RB}) or
  equivalently~$RB(P_0,P_0) = 0$. Since the latter is a symmetric bilinear form
  and~$\bfk[x]$ has characteristic zero, the polarization identity implies that
  it suffices to prove~$RB(P_0, P_0)(u,u) = 0$ for all~$u \in \bfk[x]$. Of
  course we may restrict ourselves to the canonical basis~$u = x^n$, so it
  remains to show $P(x^{\sigma(n)})^2 = 2 P_0(x^n \,
  P(x^{\sigma(n)}))$. Applying the definition of~$P$, this is equivalent to
  \begin{equation*}
    \beta(\sigma(n))^2 \, x^{2\theta(\sigma(n))} =
    2 \beta(\sigma(n)) \, P_0(x^{n+\theta(\sigma(n))}),
  \end{equation*}
  and we may use Eq.~(\mref{eq:sigma}) to expand the right-hand side further to
  \begin{equation*}
    2 \beta(\sigma(n)) \, \beta(\sigma(n)+\theta(\sigma(n))) \,
    x^{\theta(\sigma(n)+\theta(\sigma(n)))}.
  \end{equation*}
  But now we may apply Eqs.~(\mref{eq:dege1}) and~(\mref{eq:dege2}) of
  Theorem~\mref{thm:gen} to conclude that this is equal to the left-hand
  side. Hence~$P_0$ is indeed a monomial RBO on~$\bfk[x]$. Clearly~$P_0(x^n) =
  P(x^{\sigma(n)}) \neq 0$ since~$\sigma(n) \in \supp_\beta$, so~$P_0$ is
  nondegenerate.
\smallskip

\noindent
(\ref{it:degcls2})
For defining~$s_1, \dots, s_k$, consider first the sets~$\Sigma_i :=
\supp_\beta \cap (i + e \NN)$ for~$i \in \{0, \dots, e-1\}$. Suppressing the
empty ones, we reindex the rest as~$\Sigma_1, \dots, \Sigma_k$ for~$1 \le k \le
e$. Then for any~$i \in \{ 1, \dots, k \}$ there exists~$\sigma_i \in \Sigma_i$
such that~$\sigma_i + e \NN \subseteq \Sigma_i$. Indeed, one may
choose~$\sigma_i = \sigma_i' + (f+1)e$ for any~$\sigma_i' \in \Sigma_i$ since
then~$(f+1)e \in T$, and the hypothesis~$\supp_\beta+T\subseteq \supp_\beta$
implies the required condition~$\sigma_i + e \NN \subseteq \Sigma_i$. Let~$s_i
\in \Sigma_i$ be minimal such that the condition is satisfied; this implies in
particular~$s_i - e \not\in \supp_\beta$. Then clearly~$\Sigma_i = C_i \uplus
(s_i + e \NN)$ for finite sets~$C_i \subset \NN$. Now define~$C := C_1 \cup
\cdots \cup C_k$ to obtain the decomposition~(\mref{eq:beta-shape}). We must
have~$k<e$ since otherwise~$\nul_\beta \subseteq \{ 0, \dots, \max(s_1,
\dots, s_e)\}$ is finite, contradicting Lemma~\mref{lem:fin-or-cofin}.
Finally, note that~$E := \supp_\beta \setminus (\supp_\beta+T)$ is bounded
by~$\max(s_1, \dots, s_k)+(f+1)e$ and hence finite. Clearly, $\theta$
is determined on~$\supp_\beta \setminus E$ by Eq.~(\mref{eq:dege1}).
\end{proof}

\section{Injective Rota-Baxter operators on \texorpdfstring{$\bfk[x]$}{bkfx}}
\mlabel{sec:inj}

For now let~$\bfk$ be an arbitrary field of characteristic zero. An
important subclass of Rota-Baxter operators~$P$ on~$\bfk[x]$ are those
associated with the standard derivation~$\partial$ in the sense
that~$\partial \circ P = 1_{\bfk[x]}$. We generalize this for
arbitrary~$r \in \nonz{\bfk[x]}$ to the \emph{differential
  law}~$\partial \circ P = r$, where~$r$ denotes the corresponding
multiplication operator. Thus we define
\begin{equation}
 \rbo_r(\bfk[x]):=\{ P\in \rbo(\bfk[x])\,|\, \partial \circ P = r\}.
\mlabel{eq:rbor}
\end{equation}
Let us now show that the class of all operators satisfying a
differential law actually coincides with the class of all injective
operators, which we denote by~$\rbo_*(\bfk[x])$.
\begin{theorem}
We have~$\rbo_*(\bfk[x]) =
  \bigcup_{r \in \nonz{\bfk[x]}} \rbo_r(\bfk[x])$.
  \mlabel{thm:inj-diff-law}
\end{theorem}
\begin{proof}
  The inclusion from right to left is simple as~$P(f) = 0$
  implies~$\partial(P(f)) = rf = 0$ and hence~$f=0$ since~$\bfk[x]$ is
  an integral domain.

  Now let~$P\colon \bfk[x] \to \bfk[x]$ be an injective Rota-Baxter
  operator. Then there exists a linear map~$D\colon \im(P) \to
  \bfk[x]$ with~$D \circ P = 1_{\bfk[x]}$. Adjoining~$\bfk$ as
  constants, one can immediately check that~$D$ is a derivation on the
  unitary subalgebra~$J := \bfk + \im(P)$. Note that~$P(1) \not\in
  \bfk$ since~$P(1) = c$ implies~$c^2 = P(1)^2 = 2 P(P(1)) = 2 c^2$
  and hence~$c=0$, contradicting injectivity. This means~$\bfk
  \subsetneq J$. Since~$J \subseteq \bfk[x]$ is an integral domain,
  $D$ extends uniquely to a derivation on the fraction field~$K
  \subseteq \bfk(x)$ of the ring~$J$. By L\"uroth's
  theorem~\mcite[Thm.~11.3.4]{PMCo}, the intermediate field~$\bfk
  \subset K \subseteq \bfk(x)$ is a simple transcendental extension
  of~$\bfk$, so there exists~$\phi \in \bfk(x) \setminus \bfk$ with~$K
  = \bfk(\phi)$. But then~$K \subseteq \bfk(x)$ is an algebraic field
  extension~\mcite[\S73]{vW}, so the derivation~$D$ extends uniquely
  to~$\bfk(x)$ according to~\mcite[Thm.~11.5.3]{PMCo}. But it is well
  known~\mcite[Prop.~1.3.2]{Now} that every $\bfk$-derivation
  on~$\bfk[x]$ is a multiple of the canonical derivation, so we must
  have~$D = \psi \partial$ for~$\psi := D(x)$. Then~$D \circ P = 1$
  on~$\bfk[x]$ implies that~$\psi \cdot P(1)' = 1$, so we obtain~$D =
  r^{-1} \partial$ with~$r := P(1)' \in \bfk[x]$ and then
  also~$\partial \circ P = r$.
\end{proof}

Thus the study of injective Rota-Baxter operators on $\bfk[x]$ reduces
to the study of $\rbo_r(\bfk[x])$. As noted in
Section~\mref{sec:basis}, all standard integral operators~$\stdint{a}$
are in~$\rbo_1(\bfk[x])$; more generally, the  analytically modelled operators~$\stdint{a}r$ are in~$\rbo_r(\bfk[x])$. It is thus tempting
to speculate that~$\rbo_r(\bfk[x])$ is exhausted by
the~$\stdint{a}r$. For the special case~$\bfk = \RR$ and~$r = x^k$
this will be proved at the end of this section in
Theorem~\mref{thm:main}.  For the moment, let~$\bfk$ be an arbitrary
field containing~$\QQ$.

From integration over the reals, it is well known that the difference
between two indefinite integrals is always a definite integral, which
may be interpreted as a \emph{measure}. This generalizes to the
algebraic setting in the following way.

\begin{lemma}
  Let~$r \in \nonz{\bfk[x]}$ and~$a \in \bfk$ be arbitrary. Then $P
  \in \End(\bfk[x])$ satisfies the differential law $\partial \circ
  P=r$ if and only if~$\stdint{a}r-P \in \bfk[x]^{\ast}$.
  \mlabel{lem:image}
\end{lemma}
\begin{proof}
  Since $\partial \circ \stdint{a}r=r$, a linear operator $P$ on
  $\bfk[x]$ satisfies $\partial \circ P=r$ if and only if $\partial \circ
  \mu = 0$ for $\mu := \stdint{a}r-P$. The latter identity holds if
  and only if $\im(\mu)$ is contained in~$\ker(\partial) = \bfk$.
\end{proof}

In analogy to the reals, we call the above linear functional~$\mu$ the
\emph{associated measure} of~$P$. Then the lemma says that the linear
operators satisfying the differential law are classified by their
associated measures in the sense that
\begin{equation*}
  \{P \in \End(\bfk[x]) \mid \partial \circ P = r \} = \stdint{a}r + \bfk[x]^\ast,
\end{equation*}
where the initialized point~$a$ may be chosen arbitrarily
(typically~$a=0$). But in the real case, a measure is more than an
arbitrary linear functional; for the algebraic situation this is
captured in the following result. Here and henceforth we employ the
abbreviation~$\star_{r,a}$ for~$\star_{\stdint{a}r}$,
and~$\star_r$ for~$\star_{r,0}$.

\begin{theorem}
  Let~$r \in \nonz{\bfk[x]}$ and~$a \in \bfk$ be arbitrary. Then the
  map defined by~$P\mapsto \stdint{a}r-P$ is a bijection between
 $\rbo_r(\bfk[x])$ and $\alghom{(\bfk[x],\star_{r,a})}$.
  \mlabel{thm:bij}
\end{theorem}
\begin{proof}
By Lemma~\mref{lem:image} and Proposition~\mref{pp:compcons}(\mref{it:fun}), we obtain an surjective map
$$\rbo_r(\bfk[x])\to \alghom{(\bfk[x],\star_{r,a})},\qquad
P\mapsto \stdint{a}r-P.$$ The map is injective since
$\stdint{a}r-P=\stdint{a}r-\tilde{P}$ implies $P=\tilde{P}$.
\end{proof}

Thus the above classification of
operators satisfying differential law may be refined to
$$\rbo_r(\bfk[x]) = \stdint{a}r - \alghom{(\bfk[x],\star_{r,a})}.$$ For
working out a more explicit description, we specialize to the monomial
case~$r = x^k$,
where we use the abbreviation~$\star_k$
for~$\star_{x^k}$. To this end, we will
determine~$\alghom{(\bfk[x],\star_k)}$, starting with~$k=0$.

\begin{theorem}
\begin{thmenumerate}
\item For any~$k \in \NN$, we have the
  isomorphism~$(\bfk[x],\star_{k}) \cong x^{k+1}\bfk[x]$ of nonunitary
  algebras.  \mlabel{it:iso}
\item There is a bijection $\alghom{(\bfk[x],\star_{0})} \to \bfk$
  that sends~$\mu$ to~$\mu(1)$. In particular, the value~$a := \mu(1)
  \in\bfk$ determines~$\mu$ uniquely by
\begin{equation}
\mu(x^n)=\frac{1}{n+1} \, a^{n+1}
\mlabel{eq:power}
\end{equation}
for all $n\in\NN$. Moreover, the codimension of $\ker(\mu)$ equals~$0$ for~$a=0$,
and~$1$ for~$a \neq 0$.  \mlabel{it:bij}
\end{thmenumerate}
\mlabel{thm:uniquev}
\end{theorem}

\begin{proof}
  (\mref{it:iso}) Note that $\{u_n:=nx^{n-k-1}\,|\,n\geq k+1\}$ is
  a $\bfk$-linear basis of $\bfk[x]$ with
\begin{align*}
 u_m\star_{k} u_n
& = mx^{m-k-1} \stdint{0}(x^k \cdot nx^{n-k-1}) + nx^{n-k-1}
\stdint{0}(x^k \cdot mx^{m-k-1})\\
& = mx^{m-k-1}x^n + nx^{n-k-1}x^m
=(m+n)x^{m+n-k-1}=u_{m+n}.
\end{align*}
Thus the $\bfk$-linear map induced by~$u_n \mapsto x^n \; (n \geq
k+1)$ is an isomorphism~$(\bfk[x],\star_{k})\to x^{k+1}\bfk[x]$
of nonunitary $\bfk$-algebras as claimed.  \smallskip

\noindent
(\mref{it:bij}) Since $x\bfk[x]$ is the free nonunitary commutative
$\bfk$-algebra on~$x$, so is $(\bfk[x],\star_0)$ by the isomorphism
from~(\mref{it:iso}). Then the bijection follows from the universal
property of free nonunitary commutative~$\bfk$-algebra on~$x$. Note
that under the isomorphism from~(\mref{it:iso}), the generator~$x$
of~$x\bfk[x]$ corresponds to the generator~$1 = u_1$ of~$(\bfk[x],
\star_0)$.

To prove Eq.~(\mref{eq:power}), we use induction on $n$. For the base
case~$n=0$, we have $\mu(1)=a$ by the definition of~$a$. Now
suppose Eq.~(\mref{eq:power}) has been proved for a fixed~$n$.  Since
$$1\star_{0} x^n=\stdint{0}(x^n) + x^n \, \stdint{0}(1)=\frac{n+2}{n+1}
\, x^{n+1}$$ and $\mu$ is an $\bfk$-algebra homomorphism, we have
$$\mu\left(\frac{n+2}{n+1}x^{n+1}\right)
=\mu(1\star_{0}x^n)=\mu(1) \, \mu(x^n) = \frac{1}{n+1} \, a^{n+2},$$ applying
the induction hypothesis in the last step. Thus we
obtain~$\mu(x^{n+1})=\frac{1}{n+2} \, a^{n+2}$, and the induction is
complete. The last statement follows since the codimension of~$\ker(\mu)$ equals
the dimension of~$\im(\mu)$ and~$\mu$ is surjective if and only if~$\mu(1)\neq
0$.
\end{proof}

At this juncture, the results accumulated are sufficient for
classifying all Rota-Baxter operators~$P$ satisfying the differential
relation~$\partial \circ P = 1_{\bfk[x]}$. This is an important
special case since it states that
all \emph{indefinite integrals} are
analytically modelled.

\begin{theorem}
  We have~$\rbo_1(\bfk[x]) = \big\{\stdint{a} \mid a \in\bfk\big\}$.
\end{theorem}
\begin{proof}
  The inclusion from right to left is clear, so assume~$P \in
  \rbo_1(\bfk[x])$. By Theorem~\mref{thm:bij}, there exists~$\mu \in
  \alghom{(\bfk[x],\star_{0})}$ such that~$P=\stdint{0}-\mu$. Setting
  now~$a := \mu(1)$, Theorem~\mref{thm:uniquev} asserts
  that~$\mu(x^n)=\frac{1}{n+1} \, a^{n+1}$ for all~$n \in \NN$. Then
  $$P(x^n)=\stdint{0}(x^n)-\mu(x^n) = \frac{x^{n+1}-a^{n+1}}{n+1} = \stdint{a}(x^n)$$
  so that~$P = \stdint{a}$, and the inclusion from left to right is established.
\end{proof}

For classifying the Rota-Baxter operators~$P$ with~$\partial \circ P =
x^k \; (k > 0)$ we must determine all algebra homomorphisms~$\mu$ with
respect to the multiplication~$\star_k$. At this point, we have to
restrict ourselves to the field~$\bfk = \RR$ since we shall make use
of the order on the \emph{reals} in the next two lemmas.

\begin{lemma}
  Let $\mu\colon (\RR[x],\star_{2{\ell}+1})\rightarrow \RR$ be an
  $\RR$-algebra homomorphism with $\ell\geq 0$. Then we have~$\mu(1)
  \geq 0$.  \mlabel{lem:valueone}
\end{lemma}
\begin{proof}
  Since~$1\star_{2{\ell}+1} 1 = 2\, \stdint{0}(x^{2{\ell}+1}) =
  x^{2{\ell}+2}/(\ell+1)$ and $\mu$ is an $\RR$-algebra homomorphism,
  we obtain
  ~$c^2 = \mu(1\star_{2{\ell}+1}1) =
  \mu(x^{2{\ell}+2}/(\ell+1))$, where we have set~$c := \mu(1)$. Hence
  we get the relation~$\mu(x^{2{\ell}+2})=({\ell}+1)c^2$. We have also
\begin{equation*}
  1\star_{2{\ell}+1}x^{2{\ell}+2} = \stdint{0}(x^{4{\ell}+3})
  +x^{2{\ell}+2} \, \stdint{0}(x^{2{\ell}+1})
  = \frac{3}{4{\ell}+4} \, x^{4{\ell}+4},
\end{equation*}
which implies by the~$\RR$-algebra homomorphism property and the
previous relation that
\begin{equation}
\mu(x^{4{\ell}+4})=\frac{4}{3}({\ell}+1)^2c^3.
\mlabel{eq:nu3}
\end{equation}
Next we observe that $x^{{\ell}+1}\star_{2{\ell}+1}x^{{\ell}+1}
=2x^{{\ell}+1} \, \stdint{0}(x^{3{\ell}+2}) = (2/3) \,
x^{4{\ell}+4}/(\ell+1)$. Setting $\tilde{c} := \mu(x^{{\ell}+1})$,
this yields yet another relation
\begin{equation}
\mu(x^{4{\ell}+4})=\frac{3}{2}(\ell+1) \, \tilde{c}^2.
\mlabel{eq:nuvaluet}
\end{equation}
Combining Eqs.~(\mref{eq:nu3}) and~(\mref{eq:nuvaluet}), we obtain
$\frac{4}{3}({\ell}+1)^2c^3= \frac{3}{2} (\ell+1) \tilde{c}^2$ and
thus $c = \sqrt[3]{\frac{9}{8({\ell}+1)}\tilde{c}^2}\geq0$.
\end{proof}

\begin{lemma}
  Let $\mu\colon (\RR[x],\star_k)\rightarrow\RR$ be an $\RR$-algebra
  homomorphism for~$k\in\NN$. Then there exists a number $a \in\RR$
  such that $\mu(1) = a^{k+1}/(k+1)$.  \mlabel{lem:nuvall}
\end{lemma}
\begin{proof}
  We set~$c := \mu(1)$ and~$a :=
  \!\!\sqrt[k+1]{(k+1)c}$. If~$k=2\ell+1$ with~$\ell\in\NN$,
  Lemma~\mref{lem:valueone} implies that~$c\geq0$ and we may extract
  an even root to obtain~$a \in \RR$. If on the other hand~$k=2\ell$
  for~$\ell\in\NN$, the root in~$a$ is odd and hence clearly~$a \in
  \RR$ also in this case.
\end{proof}

The number~$a$ ensured by the previous lemma serves to characterize
the  associated measure~$\mu$ of the Rota-Baxter operator underlying the
double product~$\star_k$. Analytically speaking, Analytically speaking, $\mu(1)$ is the
Riemann integral over~$[0,a]$.

\begin{prop}
  Let $\mu\colon (\RR[x],\star_k)\rightarrow \RR$ be an $\RR$-algebra
  homomorphism. Then there exists a number~$a\in\RR$ such that
  $\mu(x^n) = a^{n+k+1}/(n+k+1)$ for all~$n\in\NN$. In
  particular, $\mu$ is uniquely determined by~$a$.
  \mlabel{prop:allhomk}
\end{prop}
\begin{proof}
  We prove the claim by induction on $n \in \NN$. In the base
  case~$n=0$, Lemma~\mref{lem:nuvall} yields $\mu(1) =
  a^{k+1}/(k+1)$. Suppose now the claim has been proved up to a
  fixed~$n$. Since
  $$1\star_k x^{n-k} = \stdint{0}(x^n) +x^{n-k} \, \stdint{0}(x^k) =
  \frac{n+k+2}{(n+1)(k+1)} \, x^{n+1}$$ and $\mu$ is an $\RR$-algebra
  homomorphism, we have
  $$\mu\bigg(\frac{n+k+2}
  {(n+1)(k+1)} \, x^{n+1}\bigg) =\mu(1 \star_kx^{n-k}) = \mu(1) \,
  \mu(x^{n-k}) = \frac{1}{(n+1)(k+1)} \, a^{n+k+2},$$ where we have
  applied the induction hypothesis in the last step since~$n-k\leq
  n$. But this gives immediately~$\mu(x^{n+1}) = a^{n+k+2}/(n+k+2)$,
  which completes the induction.
\end{proof}

Since the number~$a$ of the proposition above characterizes the
associated measures, we obtain now the \emph{desired classification}
of the Rota-Baxter operators~$P$ on~$\RR[x]$ that satisfy the
differential relation~$\partial \circ P = x^k$. The number~$a$ plays
the role of the initialization point of the integral (we regain
the standard integral~$\stdint{0}$ for~$a=0$ since then the associated
measure is zero).

\begin{theorem}
  We have~$\rbo_{x^k}(\RR[x]) = \big\{ \stdint{a}x^k \mid a \in \RR
  \big\}$ for any~$k \in \NN$.
  \mlabel{thm:main}
\end{theorem}
\begin{proof}
  The inclusion from right to left is clear, so assume~$P \in
  \rbo_{x^k}(\RR[x])$. Then Theorem~\mref{thm:bij} yields an
  $\RR$-algebra homomorphism $\mu\colon(\RR[x],\star_k)\rightarrow
  \RR$ such that $P=\stdint{0} x^k-\mu$. By
  Proposition~\mref{prop:allhomk}, there exists a number~$a\in\RR$
  such that~$\mu(x^n) = a^{n+k+1}/(n+k+1)$. Thus we have
  $$P(x^n)=\stdint{0}(x^{n+k}) - \mu(x^n) = (x^{n+k+1} -
  a^{n+k+1})/(n+k+1) = \stdint{a}(x^{n+k}),$$ so that~$P =
  \stdint{a}x^k$, and the inclusion from left to right is
  established.
\end{proof}

As mentioned earlier, it is tempting to generalize the above result
from monomials to \emph{arbitrary polynomials}. Together with
Theorem~\mref{thm:inj-diff-law}, this would imply that
\begin{equation*}
  \rbo_*(\bfk[x]) = \bigcup_{r \in \nonz{\bfk[x]}} \rbo_r(\bfk[x]) =
  \{ \stdint{a}r \mid a \in \bfk, r \in \nonz{\bfk[x]} \} ,
\end{equation*}
for the case~$\bfk = \RR$. The missing inclusion is as follows.

\begin{conjecture}
  We have~$\rbo_{r}(\RR[x]) \subseteq \big\{ \stdint{a}r \mid a \in \RR \big\}$
  for any~$r \in \nonz{\RR[x]}$.  \mlabel{conj:rbor}
\end{conjecture}

In the rest of this paper, we add some preliminary results in support
of this conjecture. Let us call a Rota-Baxter operator~$P$ on~$\RR[x]$
\emph{initialized} at a point~$a \in \RR$ if~$\evl_a \circ P$ is the
zero operator, where~$\evl_a\colon \RR[x] \to \RR[x]$ denotes
evaluation at~$a$. The typical case is when~$P = \stdint{a}r$. It is
easy to see that Conjecture~\mref{conj:rbor} is equivalent to the
claim that all Rota-Baxter operators in~$\rbo_r(\RR[x])$ are
initialized. Indeed, if~$P$ is initialized at~$a$, then we may
multiply the differential law~$\partial \circ P = r$ by~$\stdint{a}$
from the left to obtain~$P = \stdint{a}r$ since we
have~$\stdint{a} \partial = 1_{\RR[x]}-\evl_a$. So for proving
Conjecture~\mref{conj:rbor} one has to determine the initialization
point~$a$ from a given Rota-Baxter operator~$P$ and~$r \in
\nonz{\RR[x]}$. If~$P$ is already known to be of the
form~$\stdint{a}r$, this can be done as follows.

\begin{lemma}
  For the Rota-Baxter operator~$P = \stdint{a}r$ with~$a \in \RR$ and~$r \in
  \nonz{\RR[x]}$ we have
  \begin{equation}
    \mlabel{eq:init-pt}
    a = \frac{P(2xr'+r)-xr^2}{P(2r')-r^2},
  \end{equation}
  provided~$r(a) \neq 0$. On the other hand, if~$r(a) = 0$ then~$P =
  (r-\stdint{a}r') \circ \stdint{0}$.
  \mlabel{lem:init-pt}
\end{lemma}
\begin{proof}
  Let us first consider the generic case~$r(a) \neq 0$. Using the differential
  law~$\partial \circ P = r$, one sees immediately that numerator and
  denominator are both constants since they vanish under~$\partial$. Moreover,
  the denominator cannot be zero since we have
  \begin{equation*}
    P(2r') = \int_a^x (r^2)' = r^2 - r(a)^2 \neq r^2
  \end{equation*}
  by the assumption of genericity. Integrating~$(r^2 r^{(i)})' = 2 rr'r^{(i)} +
  r^2 r^{(i+1)}$ from~$a$ to~$x$, we obtain
  \begin{equation}
    \label{eq:r2ri}
    r^2 r^{(i)} - r(a)^2 r^{(i)}(a) = P(2r'r^{(i)} + rr^{(i+1)}).
  \end{equation}
  Assuming~$r$ has degree~$n$, we can write
  \begin{equation*}
    r = 1 + r_1 x + r_2 \frac{x^2}{2!} + \cdots + r_n \frac{x^n}{n!}
  \end{equation*}
  so that~$r^{(n-1)} = r_{n-1} + r_n x$ and~$r^{(n)} =
  r_n$. Substituting~$i=n-1$ and~$r(a)^2 = r^2 - P(2r')$ in Eq.~\eqref{eq:r2ri}, we
  obtain the relation
  \begin{equation}
    \mlabel{eq:det-rel}
    (r_{n-1}+r_n x) r^2 - (r^2 - P(2r')) (r_{n-1}+r_n a) = P(2 r_{n-1}r' +
    2 r_n x r' + r_nr),
  \end{equation}
  which simplifies to~$(x-a)r^2 = P(2xr'-2ar'+r)$. Solving this for~$a$
  gives Eq.~(\mref{eq:init-pt}).

  Now assume~$r(a) = 0$. Then for~$f \in \RR[x]$ we obtain
  \begin{equation*}
    Pf' = \int_a^x rf' = [rf]_a^x - \int_a^x r'f = rf - \stdint{a}r'f
  \end{equation*}
  and hence by~$(\stdint{0}f)' = f$ the required identity~$Pf = (r \stdint{0})f
  - (\stdint{a}r'\stdint{0})f = (r - \stdint{a}r')\stdint{0}(f)$.
\end{proof}

Lemma~\mref{lem:init-pt} suggests the following strategy for proving
Conjecture~\mref{conj:rbor}. Given an arbitrary~$P \in
\rbo_r(\RR[x])$, we determine first the denominator
of Eq.~(\mref{eq:init-pt}). If it vanishes, we try to find~$\tilde{P} \in
\rbo_{r'}(\RR[x])$ with~$P = (r-\tilde{P}) \circ \stdint{0}$, and we
use induction on the degree of~$r$ to handle~$\tilde{P}$. In the
generic case of non-vanishing denominator, we compute the value of~$a$
from Eq.~(\mref{eq:init-pt}), and it suffices to prove that~$P$ is
initialized at~$a$. For doing this, the first step would be to
ascertain that~$r(a)^2 = r^2-P(2r')$. This would imply that~$P(r')$
vanishes at~$x=a$ and hence also~$P(2xr'+r)$
by Eq.~\eqref{eq:det-rel}. Using the Rota-Baxter axiom and the above
relations, one can produce polynomials~$p$ such that~$P(p)$ vanishes
at~$x=a$. If this is done for sufficiently many polynomials~$p$ to
generate~$\RR[x]$ as a real vector space, we are done. Here is an
example of a class of polynomials where one can infer vanishing
at~$x=a$ provided~$r(a)^2 = r^2 -P(2r')$ has been established. For~$P
= \stdint{a}r$, it recovers the fact that~$\stdint{a}(r' r^{2k+1}) =(2k+2)^{-1}
\stdint{a}((r^{2k+2})') = (2k+2)^{-1} (r^{2k+2}-r(a)^{2k+2})$.

\begin{lemma}
  Let~$P \in \rbo_r(\RR[x])$ be arbitrary. Then we have~$P(r' r^{2k})
  = (2k+2)^{-1} (r^{2k+2} - c^{k+1})$ for~$c := r^2-P(2r')
  \in \RR$ and all~$k \ge 0$.
  \mlabel{lem:even-power}
\end{lemma}
\begin{proof}
  We use induction on~$k$. The base case~$k=0$ is immediate from the
  definition of~$c$. Now assume the claim for all degrees below a
  fixed~$k>0$; we prove it for~$k$. By the Rota-Baxter axiom
and the
  definition of~$c$ we have
  \begin{align*}
    P(r')^{k+1} &= (k+1)! \, P_{r'}^{k+1}(1) = (k+1)! \, P(r' \, P_{r'}^k(1)) =
    (k+1) \, P(r' \, P(r')^k)\\
    &= 2^{-k} (k+1) \, P(r'(r^2-c)^k),
  \end{align*}
  where~$P_{r'}\colon \RR[x] \to \RR[x]$ is defined by~$P_{r'}(p) :=
  P(r'p)$. Substituting the defining relation of~$c$ on the
  left-hand side, we obtain $(r^2-c)^{k+1} = 2 (k+1) \,
  P(r'(r^2-c)^k)$, so the binomial theorem yields
  \begin{equation*}
    (2k+2) \, P(r' r^{2k}) = (r^2-c)^{k+1} - 2(k+1)
    \sum_{l=0}^{k-1} \binom{k}{l} \, (-c)^{k-l} P(r'r^{2l}).
  \end{equation*}
  Applying the induction hypothesis leads to
  \begin{align*}
    (2k+2) \, P(r' r^{2k}) &= (r^2-c)^{k+1} - (k+1) \sum_{l=0}^{k-1}
    \binom{k}{l} \, \frac{(-c)^{k-l}}{l+1} \,
    \Big((r^2)^{l+1}-c^{l+1}\Big)\\
    &= (r^2-c)^{k+1} + (r^{2k+2}-c^{k+1}) - (k+1) \sum_{l=0}^k
    \binom{k}{l} \, \frac{(-c)^{k-l}}{l+1} \,
    \Big((r^2)^{l+1}-c^{l+1}\Big).
  \end{align*}
  For evaluating the above sum, just note that integrating~$(x-c)^k$
  from~$\alpha$ to~$\beta$ and using the binomial theorem gives
  \begin{equation*}
    \frac{(\beta-c)^{k+1} - (\alpha-c)^{k+1}}{k+1} =
    \sum_{l=0}^k \binom{k}{l}
    \, \frac{(-c)^{k-l}}{l+1} \, \Big(\beta^{l+1}-\alpha^{l+1}\Big),
  \end{equation*}
  which may be evaluated at~$(\alpha,\beta) = (c, r^2)$ in the previous sum
  to obtain
  \begin{equation*}
    (2k+2) \, P(r' r^{2k}) = (r^2-c)^{k+1} + r^{2k+2}-c^{k+1} -
   (r^2-c)^{k+1} = r^{2k+2}-c^{k+1},
  \end{equation*}
  which completes the induction.
\end{proof}

We conclude with a simple result about the double product~$\star$ in
the general case of~$\stdint{a}r$. This lemma is a kind of analogy
(though not a generalization) of
Theorem~\ref{thm:uniquev}(\ref{it:iso}).
In fact, the two results
coincide for~$r=x$.

\begin{lemma}
  Let~$\star$ be the double product corresponding to the Rota-Baxter
  operator~$\stdint{a}r$ and set~$\rho = r(a)$. Then the non-unitary
  subalgebra of~$(\bfk[x], \star)$ generated by~$u_n = n r^{n-2} r' \;
  (n \ge 2)$ is isomorphic to the non-unitary subalgebra of~$(\bfk[x],
  \cdot)$ generated by~$x^n - \rho^n \; (n \ge 2)$.
\end{lemma}
\begin{proof}
  The double product of the basis elements~$u_m \; (m \ge 2)$ and~$u_n
  \; (n \ge 2)$ is given by
  \begin{align*}
    u_m \star u_n &= mn \, r^{m-2} r' \, \stdint{a} r^{n-1}r' + mn \, r^{n-2} r' \, \stdint{a} r^{m-1}r'
    = m \, r^{m-2} r' (r^n - \rho^n) + n \, r^{n-2} r' (r^m - \rho^m)\\
    &= u_{m+n} - \rho^n \, u_m - \rho^m \, u_n,
  \end{align*}
  so the $\bfk$-linear map~$\phi$ defined by~$\phi(u_m) = x^m-\rho^m$
  is a homomorphism of nonunitary~$\bfk$-algebras since we have
  \begin{equation*}
    (x^m - \rho^m)(x^n-\rho^n) = (x^{m+n}-\rho^{m+n}) - \rho^n (x^m-\rho^m) - \rho^m (x^n-\rho^n).
  \end{equation*}
  The map~$\phi$ is clearly bijective as it maps a $\bfk$-basis to a $\bfk$-basis.
\end{proof}
\smallskip

\noindent {\bf Acknowledgements}: This work was supported by the
National Science Foundation of US (Grant No. DMS~1001855), the
Engineering and Physical Sciences Research Council of UK (Grant No.
EP/I037474/1) and the National Natural Science Foundation of China
(Grant No. 11371178). L. Guo and S. Zheng also thank KITPC and Morning Center of Mathematics in Beijing for hospitality and support.

\end{document}